\documentclass[journal]{IEEEtran}
\usepackage{amssymb,epsfig,color,cite,amsmath,amsfonts,mathrsfs,algorithm,algorithmicx}
\usepackage{epstopdf}
\usepackage{datetime,fancyhdr}
\usepackage{algpseudocode}
\usepackage{arydshln}
\usepackage{multirow}
\usepackage{amsthm} 

\usepackage{times} 
\usepackage{epsfig}
\usepackage{graphicx}
\usepackage{latexsym}
\usepackage{subfigure}
\usepackage{hyperref}
\newtheorem{lemma}{Lemma}[section]
\newtheorem{theorem}{Theorem}[section]

\newtheorem{remark}{Remark}[section]
\newtheorem{definition}{Definition}[section]

\newtheorem{proposition}{Proposition}[section]
\newtheorem{assumption}{Assumption}[section]

\ifCLASSINFOpdf
\else
\fi
\begin{document}
%
\title{ Distributed stochastic proximal algorithm with random reshuffling for non-smooth finite-sum optimization}
%
%
%

\author{Xia Jiang,
        Xianlin~Zeng,~\IEEEmembership{Member,~IEEE,}
        Jian~Sun,~\IEEEmembership{Senior Member,~IEEE,}
        Jie~Chen,~\IEEEmembership{Fellow,~IEEE,}
        and~Lihua~Xie,~\IEEEmembership{Fellow,~IEEE} 
\thanks{This work was supported in part by the National Key R$\&$D Program of China under Grant 2021YFB1714800, the National Natural Science Foundation of China under Grants 61925303, 62088101, 62073035, 62173034, the Natural Science Foundation of Chongqing under Grant 2021ZX4100027 and China Scholarship Council 202006030072. \emph{(Corresponding author: Jian Sun.)}}
\thanks{X. Jiang (jiang-xia@bit.edu.cn) and J. Sun (sunjian@bit.edu.cn) are with Key Laboratory of Intelligent Control and Decision of Complex Systems, School of Automation, Beijing Institute of Technology, Beijing, 100081, China, and also with the Beijing Institute of Technology Chongqing Innovation Center, Chongqing  401120, China}
\thanks{X. Zeng (xianlin.zeng@bit.edu.cn) is with Key Laboratory of Intelligent Control and Decision of Complex Systems, School of Automation, Beijing Institute of Technology, Beijing, 100081, China}
\thanks{J. Chen (chenjie@bit.edu.cn) is with the School of Electronic and Information Engineering, Tongji University, Shanghai, 200082, China, and also with Beijing Advanced Innovation Center for Intelligent Robots and Systems (Beijing Institute of Technology), Key Laboratory of Biomimetic Robots and Systems (Beijing Institute of Technology), Ministry of Education, Beijing, 100081, China.}
\thanks{L. Xie (elhxie@ntu.edu.sg) is with the School of Electrical and Electronic Engineering, Nanyang Technological University, Singapore 639798.}}

\maketitle

{
\begin{abstract}
 The non-smooth finite-sum minimization is a fundamental problem in machine learning. This paper develops a distributed stochastic proximal-gradient algorithm with random reshuffling to solve the finite-sum minimization over time-varying multi-agent networks. The objective function is a sum of differentiable convex functions and non-smooth regularization. Each agent in the network updates local variables by local information exchange and cooperates to seek an optimal solution. We prove that local variable estimates generated by the proposed algorithm achieve consensus and are attracted to a neighborhood of the optimal solution with an
  $\mathcal{O}(\frac{1}{T}+\frac{1}{\sqrt{T}})$ convergence rate, where $T$ is the total number of iterations.  Finally, some comparative simulations are provided to verify the convergence performance of the proposed algorithm. 
\end{abstract}
}

%
\IEEEpeerreviewmaketitle

\section{Introduction}
%
%
%
%
\IEEEPARstart{D}{istributed} non-smooth finite-sum minimization is a basic problem of the popular supervised machine learning \cite{no_smooth_view,non-smooth_nnls,svm_nonsmooth}. In machine learning communities, optimizing a training model with merely the average loss over a finite data set usually leads to over-fitting or poor generalization. Some non-smooth regularizations are often included in the cost function to encode prior knowledge, which also introduces the challenge of non-smoothness. What's more, in many practical network systems \cite{distri_tra_coord,dis_ecodis_powerflow}, training data are naturally stored at different physical nodes and it is expensive to collect data and train the model in one centralized node. Compared to the centralized setting, the distributed setting makes use of multi computational sources to train the learning model in parallel, leading to potential speedup. However, since the data are distributed and the communication is limited, more involved approaches are needed to solve the minimization problem. Therefore, distributed finite-sum minimization has attracted much attention in machine learning, especially for large-scale applications with big data. 
\par For large-scale non-smooth finite-sum optimization problems, there have been many efficient centralized and distributed first-order algorithms, including subgradient-based methods \cite{sto_subgra_ctt,DGM-nedic,pmlr-v33-zhong14} and proximal gradient algorithms \cite{Schmidt2012,bertsekas2017incremental,dis_conti_prox,dis_prox_constraint,dec_pro_linear,ADMM_pro_dec}. Subgradient-based method is very generic in optimization with the expense of slow convergence. To be specific, subgradient-based algorithms may increase the objective function of the optimization problem for small step-sizes \cite{bertsekas2017incremental}. Proximal gradient algorithms are considered more stable than subgradient-based methods and often own better numerical performance than subgradient-based methods for convex optimization \cite{lec}. Hence, proximal gradient algorithms have attracted great interest for large-scale non-smooth optimization problems. Particularly,  training data are allocated to different computing nodes and the non-smooth finite-sum optimization needs to be handled efficiently in a distributed manner. Many distributed deterministic proximal gradient algorithms have been proposed with guaranteed convergence for non-smooth optimization over time-invariant and time-varying graphs \cite{dis_prox_constraint,dec_pro_linear,ADMM_pro_dec}. These distributed works are developed under the primal-dual framework and are applicable for the constrained non-smooth optimization. However, these distributed algorithms focus on the deterministic non-smooth optimization and require computing the full gradients of all local functions at each iteration. The per iteration computational cost  of a deterministic algorithm is much higher than that of a stochastic gradient method, which hinders the application of deterministic algorithms for non-smooth optimization with large-scale data.  
\par In large-scale machine learning settings, various first-order stochastic methods are leading algorithms due to their scalability and low computational requirements.  Decentralized stochastic gradient descent algorithms have gained a lot of attention recently, especially for the traditional federated learning setting with a star-shaped network topology  \cite{clus_fl_nnls,gradient_des_if,cc_dis_sto}. All of these methods are not fully distributed in a sense that they require a central parameter server. The potential bottleneck of the star-shaped network is possible communication traffic jam on the central sever and the performance will be significantly degraded when the network bandwidth is low. To consider a more general distributed network topology without a central server, many distributed stochastic gradient algorithms have been studied \cite{sgd_nonsmooth_opti,per_dis_sto,dsg_communi,DPSG,Cons_based_dis} for convex finite-sum optimization. 
To avoid the degenerated performance of algorithms with diminishing steps-sizes, some recent works \cite{DPSG} and \cite{Cons_based_dis} have proposed consensus-based distributed stochastic gradient descent methods with non-diminishing (constant) step-sizes for smooth convex optimization over time-invariant networks. With the variance reduction technique to handle the variance of the local stochastic gradients at each node, some distributed efficient stochastic algorithms with constant step-sizes have been studied for strongly convex smooth optimization over time-invariant undirected networks \cite{2020TSP} and time-invariant directed networks \cite{push_saga}.  However, most of the existing distributed stochastic algorithms are designed for smooth optimization and are only applicable over time-invariant networks. 
\par In practice, a communication network may be time-varying, since network attacks may cause failures of communications and connectivity of a network may change, especially for mobile networks \cite{dist_attack,distri_tra_coord}.  Therefore, it is necessary to study distributed algorithms over time-varying networks. 
Considering the unreliable network communication, some works have studied distributed asynchronous stochastic algorithms with diminishing step-sizes for smooth convex optimization over multi-agent networks \cite{per_dis_sto,dsg_communi}. To avoid the degenerated performance of algorithms with diminishing step-sizes, \cite{DGM-nedic} developed a distributed gradient algorithm with a multi-step consensus mapping for multi-agent optimization. This multi-step consensus mapping has also been widely applied in the follow-up works \cite{distri-nedich,proximal-convex} for multi-agent consensus and non-smooth optimization over time-varying networks. Differently from these deterministic works, this paper studies a distributed stochastic gradient algorithm with sample-without-replacement heuristic for non-smooth finite-sum optimization over time-varying multi-agent networks. 
\par Random reshuffling (RR) is a simple, popular but elusive stochastic method for finite-sum minimization in machine learning. Contrasted with stochastic gradient descent (SGD), where the training data are sampled uniformly with replacement, the sampling-without-replacement RR method  learns from each data point in each epoch and often own better performance in many practical problems\cite{SGD_opti_rate,why_rand_good,psg}. The convergence properties of SGD are well-understood with tight lower and upper bounds in many settings \cite{gd_2012,pmlr-v119-drori20a}, however, the theoretical analysis of RR had not been studied until recent years. The sampling-without-replacement in RR introduces a significant complication that the gradients are now biased, which implies that a single iteration does not approximate a full gradient descent step.  The incremental gradient algorithm (IG) is one special case of the shuffling algorithm. The IG generates a deterministic permutation before the start of epochs and reuses the permutation in all subsequent epochs. In contrast, in the RR, a new permutation is generated at the beginning of each epoch. The implementation of IG has a challenge of choosing a suitable permutation for cycling through the iterations \cite{Nedic_2001_siam}, and the IG is susceptible to bad orderings compared to RR \cite{Bertsekas_incre_book}. Therefore, it is necessary to study a random reshuffling algorithm for convex non-smooth optimization. For strongly convex optimization, the seminal work \cite{why_rand_good} theoretically characterized various convergence rates for the random reshuffling method.
One inspiring recent work by Richtarik \cite{NEURIPS2020_c8cc6e90} has provided some involved yet insightful proofs for the convergence behavior of RR under weak assumptions. {For smooth finite-sum minimization, \cite{MENG201946} has studied decentralized stochastic gradient with shuffling and provided insights for practical data processing.} What's more, the recent work \cite{mishchenko2021proximal} has studied centralized proximal gradient descent algorithms with random reshuffling (Prox-RR) for non-smooth finite-sum optimization and extended the Prox-RR to federated learning. However, these works are not applicable to fully distributed settings and the Prox-RR needs the strong convexity assumption, which hinders its application.
{
 \par Inspired by the existing deterministic and stochastic algorithms, we develop a distributed stochastic proximal algorithm with random reshuffling (DPG-RR) for non-smooth finite-sum optimization over time-varying multi-agent networks, extending the random reshuffling to distributed non-smooth convex optimization. 
\par The contributions of this paper are summarized as follows. 
\begin{itemize}
	\item For non-smooth convex optimization, we propose a distributed stochastic algorithm DPG-RR over time-varying networks. Although there are a large number of works on proximal SGD \cite{psg_z,NEURIPS2018_PSG_NON}, few works have studied how to extend random reshuffling to solve convex optimization with non-smooth regularization. This paper extends the recent centralized proximal algorithm Prox-RR in \cite{mishchenko2021proximal} for strongly convex non-smooth optimization to general convex optimization and distributed settings. Another very related decentralized work is \cite{MENG201946}, which has provided important insights to distributed stochastic gradient descent with data shuffling under master-slave frameworks. This paper extends \cite{MENG201946} to fully distributed settings and makes it applicable for non-smooth optimization. To our best knowledge, this is the first attempt to study fully distributed stochastic proximal algorithms with random reshuffling over time-varying networks.
	\item The proposed algorithm DPG-RR needs fewer proximal operators compared with proximal SGD and is applicable for time-varying networks.   To be specific, for the non-smooth optimization with $n$ local functions, the standard proximal SGD applies the proximal operator after each stochastic gradient step \cite{NEURIPS2018_PSG_NON} and needs $n$ proximal evaluations. In contrast, the proposed algorithm DPG-RR only operates one single proximal evaluation at each iteration. In addition, with the design of multi-step consensus mapping, this proposed algorithm enables local variable estimates closer to each other, which is crucial for distributed non-smooth optimization over time-varying networks.
	Compared with the existing incremental gradient  algorithm in \cite{Vanli_2016_CDC}, the proposed random reshuffling algorithm in this paper requires a weaker assumption, where the cost function does not need to be strongly-convex.
	\item This work provides a complete and rigorous theoretical convergence analysis for the proposed DPG-RR over time-varying multi-agent networks. Due to the sampling-without-replacement mechanism of random reshuffling, the local stochastic gradients are biased estimators of the full gradient. In addition, because of scattered information in the distributed setting, there are differences between local and global objective functions, and thus local and the global gradients. To handle the differences between local and global gradients, this paper proves the summability of the error sequences of the estimated gradients. Then, we prove the convergence of the transformed stochastic algorithm by some novel proof techniques inspired by \cite{NEURIPS2020_c8cc6e90}. 
\end{itemize}
}
 \par The remainder of the paper is organized as follows.
The preliminary mathematical notations and proximal operator are introduced in section \ref{preliminaries_sec}. The problem description and the design of a distributed solver are provided in section \ref{solver_design}. The convergence performance of the proposed algorithm is analyzed in section \ref{proof_sec}. Numerical simulations are provided in section \ref{simulation} and the conclusion is made in section \ref{conclusion}.

\section{Notations and Preliminaries} \label{preliminaries_sec}
\subsection{Notations}
\par We denote $\mathbb{R}$ the set of real numbers, $\mathbb{N}$ the set of natural numbers, $\mathbb{R}^n$ the set of $n$-dimensional real column vectors and $\mathbb{R}^{n\times m}$ the set of $n$-by-$m$ real matrices, respectively. We denote $v'$ the transpose of vector $v$. In addition, $\left\|\cdot \right\|$ denotes the Euclidean norm, $\langle\cdot,\cdot\rangle$ the inner product, which is defined by $\langle a,b\rangle=a'b$. The vectors in this paper are column vectors unless otherwise stated. The notation $[m]$ denotes the set $\{1,\cdots,m\}$. For a differentiable function $f:\mathbb{R}^n \to \mathbb{R}$, $\nabla f(x)$ denotes the gradient of function $f$ with respect to $x$. For a convex non-smooth function $h:\mathbb{R}^n \to \mathbb{R}$, $\partial h(x)$ denotes the subdifferential of function $h$ at $x$. The $\varepsilon$-subdifferential of a non-smooth convex function $h$ at ${x}$, denoted by $\partial_{\varepsilon} h(x)$, is the set of vectors $y$ such that for all $q$,
\begin{align}\label{part_phi}
h(q)-h({x})\geq y'(q-{x})-\varepsilon.
\end{align} 
\subsection{Proximal Operator}
\par For a proper non-differentiable convex function $h:\mathbb{R}^n\to (-\infty,\infty]$ and a scalar $\alpha>0$, the proximal operator is defined as
\begin{align}\label{prox_ope}
	{\rm prox}_{\alpha,h}(x)={\rm argmin}_{z\in \mathbb{R}^n}{h(z)+\frac{1}{2\alpha}\|z-x\|^2}.
\end{align} 
\par The minimum is attained at a unique point $y= {\rm prox}_{\alpha,h}(x)$, which means the proximal operator is a single-valued map. In addition, it follows from the optimality condition for convex optimization that
\begin{align}
	0\in \partial h(y)+\frac{1}{\alpha}(y-x).
\end{align} 
 The following proposition presents some properties of the proximal operator.
\begin{proposition}\cite{prox_pro}\label{prox_proposition}
	Let $h:\mathbb{R}^n\to (-\infty,\infty]$ be a closed proper convex function. For a scalar $\alpha>0$ and $x\in \mathbb{R}^n$, let $y={\rm prox}_{\alpha,h}(x)$. For $x,\hat{x}\in \mathbb{R}^n$, $\|{\rm prox}_{\alpha,h}(x)-{\rm prox}_{\alpha,h}(\hat{x})\|\leq \|x-\hat{x}\|$, which is also called the nonexpansiveness of proximal operator.
\end{proposition}
{When there exists error $\varepsilon$ in the computation of proximal operator, we denote the inexact proximal operator by ${\rm prox}_{\alpha,h}^{\varepsilon}(\cdot)$, which is defined as}
\begin{align}\label{inprox_def}
 {\rm prox}_{\alpha,h}^{\varepsilon}(x)\triangleq \Big\{\tilde{x}\big|&\frac{1}{2\alpha}\|\tilde{x}-x\|^2+h(\tilde{x})\leq \notag\\
 &\varepsilon
+\min_{z\in \mathbb{R}^d}\big\{\frac{1}{2\alpha}\|z-x\|^2+h(z)\big\}\Big\}.
\end{align}

\section{Problem Description And Algorithm Design}\label{solver_design}
{
In this paper, we aim to solve the following non-smooth finite-sum optimization problem over a multi-agent network,
\begin{align}\label{opti_pro}
\min_{x\in\mathbb{R}^d} F(x), \ F(x) =f(x)+\phi(x)=\!\frac{1}{m}\!\sum_{j=1}^m \sum_{i=1}^{n}f_{j,i}(x)+\phi(x)
\end{align}
where $x\in \mathbb{R}^d$ is the unknown decision variable, $m$ is the number of agents in the network and $n$ is the number of local samples. The global sample size of the multi-agent network is $N=mn$. The cost function $f(x)$ in \eqref{opti_pro} denotes the global cost of the multi-agent network and is differentiable. The regularization $\phi$ is non-smooth, which encodes some prior knowledge. In the multi-agent network, each agent $j$ only knows $n$ local samples and the corresponding cost functions $f_{j,i}(\cdot)$ to update the local decision variable. In addition, the non-smooth regularization $\phi$ is available to all agents. In this setting, we assume that each agent uses local cost functions and communications with its neighbors to obtain the same optimal solution $x_*$ of problem \eqref{opti_pro}.
\par In real-world applications of distributed settings, it is difficult to keep all communication links between agents connected and stable due to the the existence of networked attacks and limited bandwidth.} We consider a general time-varying multi-agent communication network $\mathcal{G}(t)=([m],\mathcal{E}(t), A(t))$, where $\mathcal{E}(t)$ denotes the set of communication edges at time $t$. At each time $t$, agent $j$ can only communicate with agent $i$ if and only if $(j,i)\in \mathcal{E}(t)$. In addition, the neighbors of agent $j$ at time $t$ are the agents that agent $j$ can communicate with at time $t$. The adjacent matrix $A(t)$ of the network is a square $m\times m$ matrix, whose elements indicate the weights of the edges in $\mathcal{E}(t)$. 
\par For the finite-sum optimization \eqref{opti_pro}, we make the following assumption.
\begin{assumption}\label{f_assump}
	\begin{itemize}
		\item[(a)] Functions $f$ and $\phi$ are convex.
		\item[(b)]Local function $f_{j,i}$ has a Lipschitz-continuous gradient
		with constant $L > 0$, i.e. for every $x,y\in \mathbb{R}^d$,
		\begin{align}\label{lcon}
		\|\nabla f_{j,i}(x)-\nabla f_{j,i}(y)\|\leq L \|x-y\|,
		\end{align}
		 while the regular function $\phi$ is non-smooth.
		\item[(c)] There exists a scalar $G_{\phi}$ such that for each sub-gradient $z\in \partial \phi(x)$, $\|z\|<G_{\phi}$, for every $x\in \mathbb{R}^d$.
		\item[(d)] There exists a scalar $G_{f}$ such that for each agent $j$ and every $x\in \mathbb{R}^d$, $\| \nabla f_{j,i}(x)\|<G_{f}$.
		\item[(e)] Further, $F$ is lower bounded by some $F_*\in \mathbb{R}$ and the optimization problem owns at least one optimal solution $x_*$.
	\end{itemize}
\end{assumption}
Assumption \ref{f_assump} is common in stochastic algorithm research \cite{MENG201946,Li2020On} and finite-sum minimization problems satisfying Assumption \ref{f_assump} arise in many applications. Here, we provide some real-world examples.
\par Example 1: In binary classification problem \cite{log_simu}, the logistic regression optimization aims to obtain a predictor $x$ to estimate the categorical variables of the testing data. When the large-scale training samples are allocated to $m$ different agents, the local cost function in \eqref{opti_pro} is a 
cross-entropy error function (CNF), $f_{j,i}(x)= \ln (1+\exp(-l_{j,i}\langle a_{j,i}, x\rangle))$, where $a_{j,i}$ and $l_{j,i}$ denote the feature vector and categorical value of the $i$th training sample at agent $j$. Then, the gradient of the convex function $f_{j,i}$ is bounded and satisfies the Lipschitz condition in \eqref{lcon}.
\par For the communication topology, the time-varying multi-agent network of the finite-sum optimization \eqref{opti_pro} satisfies {the following assumption}.
\begin{assumption}\label{net_assum}
	Consider the undirected time-varying network $\mathcal{G}(t)$ with adjacent matrices $A(t)=[a_{ij}(t)]$, $t=1,2,\cdots$
	\begin{itemize}
		\item [(a)] For each $t$, the adjacent matrix $A(t)$ is doubly stochastic.
		\item [(b)] There exists a scalar $\eta\in (0,1)$ such that $a_{jj}(t)\geq \eta$ for all $j\in [m]$. In addition, $a_{ij}(t)\geq \eta$ if $\{i,j\}\in \mathcal E(t)$, and $a_{ij}(t)=0$ otherwise.
		\item [(c)] The time-varying graph sequence $\mathcal{G}(t)$ is uniformly connected. That is, there exists an integer $B\geq 1$ such that agent $j$ sends its information to all other agents at least once every $B$ consecutive time slots.
	\end{itemize}
\end{assumption} 

\begin{remark}
	In Assumption \ref{net_assum}, part (a) is common for undirected graphs and balanced directed graphs. Part (b) means that each agent gives non-negligible weights to the local estimate and the estimates received from neighbors. Part (c) implies that the time-varying network is able to transmit information among any pair of agents in bounded time. Assumption \ref{net_assum} is widely adopted in the existing literature \cite{DGM-nedic,proximal-convex}.
\end{remark}
Next, we design a distributed proximal gradient algorithm with random reshuffling for each agent $j\in [m]$ in the multi-agent network. At the beginning of each iteration $t$, we sample indices $\pi_t^0,\cdots,\pi_t^{n-1}$ without replacement from $\{1,\cdots, n\}$ such that the generated $\pi_t=\{\pi_t^0,\cdots,\pi_t^{n-1}\}$ is a random permutation of $\{1,\cdots,n\}$. Then, we process with $n$ inner iterations of the form
\begin{align}
	x_{j,t}^{i+1}&=x_{j,t}^i-\gamma \nabla f_{j,\pi_t^i}(x_{j,t}^i), \ i\in \{0,\cdots,n-1\}, \label{x_ji_up}
\end{align}
where $\gamma$ is a constant step-size, the superscript $i$ denotes the $i$th inner iteration, the subscripts $j$ and $t$ denote the $j$th agent and $t$th outer iteration, respectively.
\par Then, to achieve the consensus between different agents, a multi-step consensus mapping is applied as
\begin{align}
v_{j,t}&=\sum_{l=1}^m \lambda_{jl,t} x_{l,t}^n\label{v_ji_up},
\end{align}
where $\lambda_{jl,t}$ is the $(j,l)$th element of matrix $\Phi(\mathbb{T}(t)+t,\mathbb{T}(t))$ for $j,l\in\{1,\cdots,m\}$.
The notation $\mathbb{T}(t)$ is the total number of communication steps before iteration $t$, and $\Phi$ is a transition matrix, defined as 
\begin{align}
\Phi(t,s)=A(t)A(t-1)\cdots A(s+1)A(s), \quad t>s\geq 0,
\end{align}
where $A(t)=[a_{ij}(t)]_{i,j=1,\cdots,m}$ is the adjacent matrix of the multi-agent network at iteration $t$. To be specific, using vector notations $v_t=[v_{j,t}]_{j=1,\cdots,m}$ and $x_t^n=[x_{j,t}^n]_{j=1,\cdots,m}$, we rewrite \eqref{v_ji_up} as 
\begin{align*}
v_t=&\Phi(\mathbb{T}(t)+t,\mathbb{T}(t))x_t^n\\
=& A(\mathbb{T}(t)+t)A(\mathbb{T}(t)+t-1)\cdots A(\mathbb{T}(t))x_t^n.
\end{align*}
Agents perform $t$ communication steps at iteration $t$. At each communication step, agents exchange their estimates $x_{j,t}^n$ and linearly combine the received estimates using weights $A(t)$. This mapping \eqref{v_ji_up} is referred as a multi-step consensus mapping because linear combinations of estimates bring the estimates of different agents close to each other.
\par Finally, for the non-smooth regular function, we use proximal operator 
\begin{align}\label{xj_up}
x_{j,t+1}={\rm prox}_{\gamma, \phi}(v_{j,t})
\end{align}
to obtain the local variable estimate $x_{j,t+1}$ of agent $j$ at the next iteration. 
\par The proposed distributed proximal stochastic gradient algorithm with random reshuffling (DPG-RR)\footnote{We suppose the local shuffling is sufficient, i.e. the permutation after shuffling is uniformly distributed.}  is formally summarized in Algorithm \ref{pg_rr_algo}.
\begin{algorithm}[H]
	\caption{DPG-RR for agent $j\in [m]$}
	\label{pg_rr_algo}
	\begin{algorithmic}[1]  
		\State (S.1) Initialization:
		\State Step-size $\gamma>0$; initial vector $x_{j,0}\in \mathbb{R}^d$ for $j\in [m]$; number of epoches $T$.
		\State (S.2) Iterations:
		\For {$t=0,1,\cdots,T-1$}
		\State Generate a random permutation $\pi_t=(\pi_t^0,\cdots,\pi_t^{n-1})$
		\State $x_{j,t}^0=x_{j,t}$
		\For {$i=0,1,\cdots,n-1$}
		\State 
		\begin{align*}
		x_{j,t}^{i+1}=x_{j,t}^i-\gamma \nabla f_{j,\pi_t^i}(x_{j,t}^i)
		\end{align*}
		\EndFor
		\State 
		\begin{align*}
		v_{j,t}&=\sum_{l=1}^m \lambda_{jl,t} x_{l,t}^n\\
		x_{j,t+1}&={\rm prox}_{\gamma, \phi}(v_{j,t})
		\end{align*}
		\EndFor
	\end{algorithmic}
\end{algorithm}
\begin{remark}
	Compared with the standard proximal SGD, where the proximal operator is applied after each stochastic gradient descent, the proposed DPG-RR only operates one single proximal evaluation at each iteration, which is $\frac{1}{n}$ of the proximal evaluations in the proximal SGD. Therefore, the proposed DPG-RR owns lower computational burden and is applicable for the non-smooth optimization where the proximal operator is expensive to calculate. 
\end{remark}

\begin{remark}\label{sample_remark}
	 Compared with the fully random sampling procedure, which is used in the well-known stochastic gradient descent method (SGD), one main advantage of the random reshuffling procedure is its intrinsic ability to avoid cache misses when reading the data from memory, enabling a faster implementation. In addition, a random reshuffling algorithm usually converges in fewer iterations than distributed SGD\cite{why_rand_good,mishchenko2021proximal}. It is due to that the RR procedure learns from each sample in each epoch while the SGD can miss learning from some samples in any given epoch. The gradient method with the incremental sampling procedure is a special case of the random reshuffling gradient method, which has been investigated in \cite{NEURIPS2020_c8cc6e90}. 
\end{remark}
\begin{remark}
	The work in \cite{MENG201946} has studied the convergence rates for strongly convex, convex and non-convex optimization under different shuffling procedures. \cite{MENG201946} has provided a unified analysis framework for three shuffling procedures, including global shuffling, local shuffling and insufficient shuffling\footnote{If the training samples of machine learning task are centrally stored, the global shuffling is usually applied, where the entire training samples are shuffled after each epoch. Otherwise, local shuffling, where each agent randomly shuffles local training samples after each epoch, is preferred to reduce communication burden.}. Unlike the work \cite{MENG201946}, this paper focuses on non-smooth optimization with a local shuffling procedure, which is more applicable for the distributed setting. In addition, this proposed distributed DPG-RR owns a comparable convergence rate and allows non-smooth functions and time-varying communications between different agents for convex finite-sum optimization.
\end{remark}
The convergence performance of the proposed algorithm is discussed in the following theorem and the proof is provided in the next section.
\begin{theorem}\label{con_theo}
 Suppose that Assumptions \ref{f_assump}, \ref{net_assum} hold and the step-size $\gamma=M/{\sqrt{T}},\ M\leq \sqrt{6}/6Ln$.
Then, Algorithm \ref{pg_rr_algo} possesses the properties:
{
\begin{itemize}
	\item [(1)] The local variable estimates $x_{j,t}$ achieve consensus and $\lim_{t\to \infty}\|x_{j,t}-\bar{x}_t\|=0$, where $\bar{x}_t\triangleq \frac{1}{m} \sum_{j=1}^m x_{j,t}$.
	\item [(2)] In addition, $\mathbb{E}[F(\hat{x}_T)-F_*]= \mathcal{O}(\frac{1}{T}+\frac{1}{\sqrt{T}})$, where $\hat{x}_T\triangleq \frac{1}{mT}\sum_{t=1}^T \sum_{j=1}^m x_{j,t}$.
\end{itemize} 
}
\end{theorem}

\begin{remark}
	If the local sample size $n$ increases, the convergence rate of the proposed algorithm will become slower. It follows from the proof of Theorem \ref{con_theo} that the constant term in the convergence rate is proportional to the third-order polynomial of $n$. Since the proposed algorithm is distributed, we can reduce the effect of local sample size $n$ on the convergence rate by using more agents for large-scale problems. Whereas more agents may require higher communication costs over multi-agent networks. In addition, the convergence rate also depends on other factors, such as the Lipschitz constant of objective functions, the upper bounds of subgradients, the shuffling variance, and the bounded intercommunication interval. 
\end{remark}
{
\begin{remark}
	Theorem \ref{con_theo} extends that of the centralized Prox-RR in \cite{mishchenko2021proximal} to distributed settings and non-strongly convex optimization.  Compared with the popular stochastic sub-gradient method with a diminishing step-size for non-smooth optimization, whose convergence rate is $\mathcal{O}(\frac{\log(T)}{\sqrt{T}})$ \cite{sgd_nonsmooth_opti}, the proposed DPG-RR has a faster convergence rate. 
\end{remark}
}
\section{Theoretical analysis}\label{proof_sec}
In this section, we present theoretical proofs for the convergence performance of the proposed algorithm. The proof sketch includes three parts. 
\begin{itemize}
	\item[(a)] Transform the DPG-RR to an algorithm with some error sequences, and prove that the error sequences are summable\footnote{ Let $\{a_n\}$ be a sequence of real numbers. The series $\sum_{k=1}^{\infty} a_k$ is summable if and only if the sequence $x_n\triangleq \sum_{k=1}^n a_k,\ n\in \mathbb{N}$ converges.}.
	\item[(b)] Estimate the bound of the forward per-epoch deviation of the DPG-RR.
	\item[(c)] Prove the consensus property of local variables and the convergence performance in Theorem \ref{con_theo}.
\end{itemize}
Each part is discussed in detail in the subsequent subsections.
To present the transformation, we define
\begin{align*}
&\bar{x}_{t}^{i}\triangleq \frac{1}{m} \sum_{j=1}^m x_{j,t}^i  \in \mathbb{R}^d,\ \bar{x}_t\triangleq \frac{1}{m} \sum_{j=1}^m x_{j,t}\in \mathbb{R}^d, \\ &\bar{v}_{t}\triangleq \frac{1}{m} \sum_{j=1}^m v_{j,t} \in \mathbb{R}^d, \  z_{t+1}\triangleq {\rm prox}_{\gamma, \phi}(\bar{v}_t).
\end{align*} 
\subsection{Transformation of DPG-RR and the summability of error sequences}
\begin{proposition}\label{aver_up}
	Suppose Assumptions \ref{f_assump} and \ref{net_assum} hold. The average variable of Algorithm \ref{pg_rr_algo} satisfies
	\begin{align}
	\bar{x}_{t+1}&\in {\rm prox}_{\gamma, \phi}^{\varepsilon_{t+1}}\Big(\bar{x}_t-\gamma \big(\sum_{i=0}^{n-1}(\nabla \mathbf{f}_{\pi_t^i}(\bar{x}_{t}^i)+e_t^i)\big)\Big),\label{xave_up}
	\end{align}
	where $\mathbf{f}_{\pi_t^i}(x)=\frac{1}{m}\sum_{j=1}^m  f_{j,\pi_t^i}(x)$, and the error sequences $\{e_t^i\}_{t=0}^{\infty}$ and $\{\varepsilon_{t+1}\}_{t=0}^{\infty}$ satisfy
\begin{subequations}
	\begin{align}
	\|e_t^i\|&\leq \frac{L}{m} \sum_{j=1}^m \|x_{j,t}^i-\bar{x}_t^i\|,\label{norm_e}\\
	\varepsilon_{t+1}&\leq  \frac{2G_{\phi}}{m}\sum_{j=1}^m \|v_{j,t} \!-\!\bar{v}_t\|\!+\!\frac{1}{2\gamma}(\frac{1}{m}\sum_{j=1}^m \|v_{j,t} \!-\!\bar{v}_t\|)^2.\label{norm_vare}
	\end{align}	
\end{subequations}

\end{proposition}
\begin{proof}
	See Appendix \ref{center_pro}.
\end{proof}
%

Then, inspired by Section III.B of  \cite{proximal-convex}, we discuss the summabilities of error sequences $\{\gamma\|e_t\|\}$ and $\{\varepsilon_t\}$ in the following proposition. 
\begin{proposition}\label{e_sum_pro}
	Under Assumptions \ref{f_assump} and \ref{net_assum}, error sequences $\{e_{t}^i\}$ and $\{\varepsilon_t\}$ satisfy $ \left\|e_t^i\right\|\leq b_{e,t}+\gamma C_0$, ${\varepsilon_t}\leq b_{\varepsilon,t}$ and $ \sqrt{\varepsilon_t}\leq b_{\sqrt{\varepsilon},t}$, where $b_{e,t}$, $b_{\varepsilon,t}$ and $b_{\sqrt{\varepsilon},t}$ are polynomial-geometric sequences and $C_0$ is a constant.
\end{proposition}
\begin{proof}
	See Appendix \ref{summ_proof}.
\end{proof}
\subsection{Boundness of the forward per-epoch deviation}
\par Before presenting the boundness analysis, we introduce some necessary quantities for the random reshuffling technique.
\begin{definition}{ \cite{NEURIPS2020_c8cc6e90}}
	For any $i$, the quantity $D_{f_i}(x,y)\triangleq f_i(x)-f_i(y)-\langle \nabla f_i(y),x-y\rangle$ is the Bregman divergence between $x$ and $y$ associated with function $f_i$.
\end{definition}
If the function $f_i$ is $L$-smooth, then for all $x,y\in \mathbb{R}^d$, $D_{f_i}(x,y)\leq \frac{L}{2}\|x-y\|^2$. The difference between the gradients of a convex and $L$-smooth function $f_i$ satisfies
\begin{align}\label{fnablf}
\|\nabla f_i(x)-\nabla f_i(y)\|^2 \leq 2L D_{f_i}(x,y).
\end{align}
\par In addition, the forward per-epoch deviation of the proposed algorithm is introduced as follows.
\begin{definition}
	The forward per-epoch deviation of DPG-RR for the $t$-th epoch is defined as
	\begin{align}\label{vt_def}
	\mathcal{V}_t\triangleq \sum_{i=0}^{n-1}\|\bar{x}_t^i-\bar{x}_{t+1}\|^2.
	\end{align}
\end{definition}

\par Now, we show that the forward per-epoch deviation $\mathcal{V}_t$ of DPG-RR is upper bounded in the following lemma.
\begin{lemma}\label{in_V_lem}
 If Assumption \ref{f_assump} holds, then the forward per-epoch deviation $\mathcal{V}_t$ satisfies
	\begin{align}
	\mathbb{E}[\mathcal{V}_t]\leq &24\gamma^2Ln^2\sum_{i=0}^{n-1} \mathbb{E}[D_{\mathbf{f}_{\pi_t^i}}(x_*,\bar{x}_t^i)]+3\gamma^2 n^2\sigma_*^2+9\gamma^2 n G_f^2\notag\\
	&+6\gamma^2 n G_{\phi}^2+12 n^3 (\gamma^2 b_{e,t}^2+\gamma^4 C_0^2)+4n\gamma b_{\varepsilon,t},
	\end{align}
where $\mathcal{V}_t$ is defined in \eqref{vt_def}, $\sigma_*^2\triangleq \frac{1}{n}\sum_{i=1}^{n} \|\nabla \mathbf{f}_{\pi_t^i}(x_*)-\frac{1}{n}\nabla f(x_*)\|^2$, $b_{e,t}$ and $b_{\varepsilon,t}$ are defined in \eqref{bet} and \eqref{bvapt}, respectively.
\end{lemma}
\begin{proof}
	See Appendix \ref{Vt_proof}.
\end{proof}
\subsection{Proof of Theorem \ref{con_theo}}
\begin{proof}
	(1) First, we prove the consensus property of the generated local variables, i.e. $\lim_{t\to \infty} x_{j,t}\to \bar{x}_t$ for all agent $j$. By \eqref{ite_c} in Lemma \ref{ite_seq}, 
	$\sum_{j=1}^m\|x_{j,t}-\bar{x}_{t} \| \leq  2 m \Gamma \Xi^t \sum_{l=1}^m\|x_{l,t-1}^n\|$ holds. In addition, with  Lemma \ref{poly_lemma} and Lemma \ref{qbound_lemma}, we obtain that $$\sum_{t=1}^{\infty} 2 m \Gamma \Xi^t \sum_{l=1}^m\|x_{l,t-1}^n\|<\infty.$$ Then, $\sum_{t=1}^{\infty} \sum_{j=1}^m\|x_{j,t}-\bar{x}_{t} \|<\infty$. Because $\sum_{j=1}^m\|x_{j,t}-\bar{x}_{t} \|$ is nonnegative, $$\lim_{t\to \infty} \sum_{j=1}^m\|x_{j,t}-\bar{x}_{t} \|=0,$$ and hence, $\lim_{t\to \infty} \|x_{j,t}-\bar{x}_{t} \|=0$ for each agent $j$. Therefore, local variable estimates $x_{j,t}$ of all agents achieve consensus and converge to the average $\bar{x}_t$ as $t\to \infty$. 
	\par (2) Next, we consider the convergence rate of the proposed algorithm. By Lemma \ref{pk_lemma} and \eqref{xave_up}, we have
	\begin{align}
	\bar{x}_t=\gamma \bar{g}_t +\bar{x}_{t+1}+\gamma e_{t+1}+p_{t+1}+\gamma \bar{d}_{t+1},
	\end{align}
	where $e_{t+1}=\sum_{i=0}^{n-1} e_{t}^i$, $\bar{d}_{t+1}\in \partial_{\varepsilon_{t+1}} \phi (\bar{x}_{t+1})$ and $\bar{g}_t\triangleq \sum_{i=0}^{n-1} \nabla \mathbf{f}_{\pi_t^i}(\bar{x}_t^i)$.
	Then, the square of the norm between $\bar{x}_t$ and the optimal solution $x_*$ satisfies
	\begin{align}\label{inxsubstar}
	&\|\bar{x}_t-x_*\|^2\notag\\
	=&\|\gamma \bar{g}_t +\gamma {e_{t+1}}+p_{t+1}+\gamma \bar{d}_{t+1}+\bar{x}_{t+1}-x_*\|^2\notag\\
	\geq & \|\bar{x}_{t+1}-x_*\|^2+2\gamma \sum_{i=0}^{n-1}\langle \bar{x}_{t+1}-x_*, \nabla \mathbf{f}_{\pi_t^i}(\bar{x}_t^i)\rangle\notag\\
	& +2\langle  \bar{x}_{t+1}-x_*, \gamma e_{t+1}+p_{t+1}+\gamma \bar{d}_{t+1} \rangle.
	\end{align}
	For the second term in \eqref{inxsubstar}, we have for any $i$,
	\begin{align}\label{inabx}
	&\langle \bar{x}_{t+1}-x_*, \nabla \mathbf{f}_{\pi_t^i}(\bar{x}_t^i)\rangle \notag\\
	=& \mathbf{f}_{\pi_t^i}(\bar{x}_{t+1})-\mathbf{f}_{\pi_t^i}(x_*)\notag\\
	&+\mathbf{f}_{\pi_t^i}(x_*)-\mathbf{f}_{\pi_t^i}(\bar{x}_t^i)-\langle \nabla \mathbf{f}_{\pi_t^i}(\bar{x}_t^i), x_*-\bar{x}_t^i\rangle \notag\\
	&-[\mathbf{f}_{\pi_t^i}(\bar{x}_{t+1})-\mathbf{f}_{\pi_t^i}(\bar{x}_t^i)-\langle \nabla \mathbf{f}_{\pi_t^i}(\bar{x}_t^i), \bar{x}_{t+1}-\bar{x}_t^i\rangle]\notag \\
	=& [\mathbf{f}_{\pi_t^i}(\bar{x}_{t+1})-\mathbf{f}_{\pi_t^i}(x_*)] \!+\!D_{\mathbf{f}_{\pi_t^i}}(x_*,\bar{x}_t^i)\!-\!D_{\mathbf{f}_{\pi_t^i}}(\bar{x}_{t+1},\bar{x}_t^i).
	\end{align}
	Summing $\mathbf{f}_{\pi_t^i}(\bar{x}_{t+1})-\mathbf{f}_{\pi_t^i}(x_*)$ over $i$ from $0$ to $n-1$ gives
	\begin{align}\label{ff}
	\sum_{i=0}^{n-1}[\mathbf{f}_{\pi_t^i}(\bar{x}_{t+1})-\mathbf{f}_{\pi_t^i}(x_*)]= f(\bar{x}_{t+1})-f_*.
	\end{align}
	Next, we bound the third term in \eqref{inabx} with $L$-smoothness in Assumption \ref{f_assump} (b) as 
	\begin{align}\label{df}
	D_{\mathbf{f}_{\pi_t^i}}(\bar{x}_{t+1},\bar{x}_t^i)\leq \frac{L}{2} \|\bar{x}_{t+1}-\bar{x}_t^i\|^2.
	\end{align}
	By summing \eqref{df} over $i$ from $0$ to $n-1$, we obtain an upper bound of the forward deviation  $\mathcal{V}_t$ by Lemma \ref{in_V_lem} that
	\begin{align*}
	&\sum_{i=0}^{n-1} \mathbb{E}[D_{\mathbf{f}_{\pi_t^i}}(\bar{x}_{t+1},\bar{x}_t^i)]
	\leq \frac{L}{2}\mathbb{E}[\mathcal{V}_t]\\
	\leq &12\gamma^2L^2n^2\sum_{i=0}^{n-1} \mathbb{E}[D_{\mathbf{f}_{\pi_t^i}}(x_*,\bar{x}_t^i)]+3\gamma^2L \frac{n^2\sigma_*^2}{2}\\
	&+\!\frac{9}{2}\gamma^2 L n G_f^2\!+\!3\gamma^2 n L G_{\phi}^2\!+\!6 L n^3 (\gamma^2 b_{e,t}^2\!+\!\gamma^4 C_0^2)\!+\!2L n\gamma b_{\varepsilon,t}. 
	\end{align*}
	With the above inequality, we estimate the lower bound of the sum of the second and third term in \eqref{inabx} as
	\begin{align}\label{dfsub}
	&\sum_{i=0}^{n-1} \mathbb{E}[D_{\mathbf{f}_{\pi_t^i}}(x_*,\bar{x}_t^i)-D_{\mathbf{f}_{\pi_t^i}}(\bar{x}_{t+1},\bar{x}_t^i)]\notag\\
	\geq& (1-  12\gamma^2L^2n^2) \sum_{i=0}^{n-1} \mathbb{E}[D_{\mathbf{f}_{\pi_t^i}}(x_*,\bar{x}_t^i)]-3\gamma^2L \frac{n^2\sigma_*^2}{2}\notag\\
	&\!-\!\frac{9}{2}\gamma^2 L n G_f^2\!-\!3\gamma^2 n L G_{\phi}^2\!-\!6 L n^3 (\gamma^2 b_{e,t}^2\!+\!\gamma^4 C_0^2)\!-\!2L n\gamma b_{\varepsilon,t}\notag\\
	\geq &-3\gamma^2L \frac{n^2\sigma_*^2}{2}-\frac{9}{2}\gamma^2 L n G_f^2-3\gamma^2n L G_{\phi}^2\notag\\
	&-6 L n^3 (\gamma^2 b_{e,t}^2+\gamma^4 C_0^2)-2L n\gamma b_{\varepsilon,t},
	\end{align}
	where the last inequality holds since $(1-  12\gamma^2L^2n^2)\geq 0$ by $\gamma\leq \sqrt{3}/6Ln$, and $\sum_{i=0}^{n-1} D_{\mathbf{f}_{\pi_t^i}}(x_*,\bar{x}_t^i)$ is non-negative due to the convexity.
	\par By rearranging \eqref{inxsubstar}, we have
	 \begin{align}\label{tra1}
	 &\|\bar{x}_{t+1}-x_*\|^2\notag\\
	 \overset{\eqref{inabx}, \eqref{ff}}{\leq} &\|\bar{x}_{t}-x_*\|^2-2\gamma (f(\bar{x}_{t+1})-f_*)\notag\\
	 &-\sum_{i=0}^{n-1} 2\gamma (D_{\mathbf{f}_{\pi_t^i}}(x_*,\bar{x}_t^i)-D_{\mathbf{f}_{\pi_t^i}}(\bar{x}_{t+1},\bar{x}_t^i))\notag\\
	 &-2 \langle \bar{x}_{t+1}-x_*,\gamma\bar{d}_{t+1}\rangle \!- 2\langle \bar{x}_{t+1}-x_*, \gamma e_{t+1}\!+\!p_{t+1}\rangle\notag\\
	 \leq & \|\bar{x}_{t}-x_*\|^2-2\gamma (f(\bar{x}_{t+1})-f_*)\notag\\
	 &-\sum_{i=0}^{n-1} 2\gamma (D_{\mathbf{f}_{\pi_t^i}}(x_*,\bar{x}_t^i)-D_{\mathbf{f}_{\pi_t^i}}(\bar{x}_{t+1},\bar{x}_t^i))\notag\\
	 &-2 \langle \bar{x}_{t+1}-x_*,\gamma\bar{d}_{t+1}\rangle\notag\\
	 &+ 2\| \bar{x}_{t+1}-x_*\| (\gamma \|e_{t+1}\|+\|p_{t+1}\|)\notag\\
	 \overset{\eqref{bet},\eqref{sqrtbvapt}}{\leq} & \|\bar{x}_{t}-x_*\|^2\!-\!2\gamma (f(\bar{x}_{t+1})\!-\!f_*)\!-\!2 \langle \bar{x}_{t+1}-x_*,\gamma\bar{d}_{t+1}\rangle\notag\\
	 &-\sum_{i=0}^{n-1} 2\gamma (D_{\mathbf{f}_{\pi_t^i}}(x_*,\bar{x}_t^i)-D_{\mathbf{f}_{\pi_t^i}}(\bar{x}_{t+1},\bar{x}_t^i))\notag\\
	 &+ 2\| \bar{x}_{t+1}-x_*\| (n \gamma b_{e,t}+n\gamma^2C_0 +\sqrt{2\gamma}b_{\sqrt{\varepsilon},t}).	 
	 \end{align}
 	where the second inequality holds due to the Cauchy Schwarz inequality and the triangle inequality.
	\par Taking expectation of \eqref{tra1} and by \eqref{dfsub},
	\begin{align*}
	&\mathbb{E}[\|\bar{x}_{t+1}-x_*\|^2]\\
	\leq &\mathbb{E}[\|\bar{x}_{t}-x_*\|^2]-2\gamma \mathbb{E}[f(\bar{x}_{t+1})-f_*]\\
	&+9\gamma^3 L n G_f^2+3\gamma^3L n^2\sigma_*^2+6\gamma^3 n L G_{\phi}^2\\
	&+12 n^3 (\gamma^3 L b_{e,t}^2+\gamma^5 C_0^2)+4\gamma^2 L n b_{\varepsilon,t}\\
	&-2\gamma \mathbb{E}[\langle  \bar{x}_{t+1}-x_*, \bar{d}_{t+1} \rangle]+2\lambda_t\mathbb{E}[\|  \bar{x}_{t+1}-x_*\|]\\
	\leq &\mathbb{E}[\|\bar{x}_{t}-x_*\|^2]-2\gamma \mathbb{E}[f(\bar{x}_{t+1})-f_*]+9\gamma^3 L n G_f^2\\
	&+3\gamma^3L n^2\sigma_*^2+6\gamma^3 n L G_{\phi}^2\\
	&+12 n^3 (\gamma^3 L b_{e,t}^2+\gamma^5 C_0^2)+4\gamma^2 L n b_{\varepsilon,t}\\ 
	&-2\gamma \mathbb{E}[\phi(\bar{x}_{t+1})-\phi_*]+2\gamma b_{\varepsilon,t}+ 2\lambda_t\mathbb{E}[\|  \bar{x}_{t+1}-x_*\|],
	\end{align*}
	where $\lambda_t=\gamma n b_{e,t}+n\gamma^2C_0+\sqrt{2\gamma}b_{\sqrt{\varepsilon},t}$, the second inequality holds because of $\langle \bar{x}_{t+1}-x_*, \bar{d}_{t+1}\rangle+\varepsilon_{t+1} \geq \phi(\bar{x}_{t+1})-\phi_*$ and \eqref{bvapt}.
	\par Then, rearranging the above result leads to
	\begin{align*}
	&2\gamma \mathbb{E}[F(\bar{x}_{t+1})-F_*]\\
	\leq &\mathbb{E}[\|\bar{x}_{t}-x_*\|^2]-\mathbb{E}[\|\bar{x}_{t+1}-x_*\|^2]+9\gamma^3 L n G_f^2\\
	&+3\gamma^3L n^2\sigma_*^2+6\gamma^3 n L G_{\phi}^2+12 n^3 (\gamma^3 L b_{e,t}^2+\gamma^5 C_0^2)\\
	&+4\gamma^2 L n b_{\varepsilon,t} +2\gamma b_{\varepsilon,t}+ 2\lambda_t\mathbb{E}[\|\bar{x}_{t+1}-x_*\|].
	\end{align*}
	Summing over $t$ from $0$ to $T-1$,
	\begin{align}\label{fsubfstar}
	&2  \gamma  \sum_{t=0}^{T-1} \mathbb{E}[F(\bar{x}_{t+1})-F_*]\notag\\
	\leq &\|\bar{x}_{0}-x_*\|^2\!-\mathbb{E}[\|\bar{x}_{T}-x_*\|^2]\!+9T\gamma^3 L n G_f^2\!+3T\gamma^3L n^2\sigma_*^2\notag\\
	&+6T\gamma^3 n L G_{\phi}^2+12T n^3 \gamma^5  C_0^2 +12\gamma^3 L n^3 \sum_{t=0}^{T-1} b_{e,t}^2\notag\\
	&+2\gamma  (1\!+\!2\gamma nL)\sum_{t=0}^{T-1}b_{\varepsilon,t}\!+\! 2\sum_{t=0}^{T-1}\lambda_t\mathbb{E}[\|  \bar{x}_{t+1}\!-\!x_*\|].
	\end{align}
	\par Then, we need to estimate an upper bound of $\mathbb{E}[\| \bar{x}_{t+1}-x_*\|]$. By the transformation of \eqref{fsubfstar} and the fact that $\mathbb{E}[\|\bar{x}_{T}-x_*\|]^2\leq \mathbb{E}[\|\bar{x}_{T}-x_*\|^2]$,
	\begin{align*}
	&\mathbb{E}[\|\bar{x}_{T}-x_*\|]^2\\
	\leq &S_T + 2\sum_{t=0}^{T-1}\lambda_t\mathbb{E}[\|  \bar{x}_{t+1}-x_*\|].
	\end{align*}
	 where $S_T= \|\bar{x}_{0}-x_*\|^2+9T \gamma^3 L n G_f^2+3T\gamma^3L {n^2}\sigma_*^2+6T\gamma^3 n L G_{\phi}^2+12n^3 \gamma^5 T C_0^2 +12\gamma^3 L n^3 \sum_{t=0}^{T-1} b_{e,t}^2 +2\gamma  (1+2\gamma nL)\sum_{t=0}^{T-1}b_{\varepsilon,t}$. It follows from Lemma \ref{usseq} that  
	\begin{align*}
	\mathbb{E}[\|\bar{x}_{T}\!-\!x_*\|]&\leq \sum_{t=0}^{T-1}(\gamma n b_{e,t}+n\gamma^2C_0+\sqrt{2\gamma}b_{\sqrt{\varepsilon},t})\\
	&+\!\Big(S_T\!+\!\big(\sum_{t=0}^{T-1}\! (\gamma n b_{e,t}\!+\!n\gamma^2C_0\!+\!\sqrt{2\gamma}b_{\sqrt{\varepsilon},t}) \big)^2\Big)^{\frac{1}{2}},
	\end{align*}
	where the sequences $b_{e,t}$, $b_{\varepsilon,t}$ and $b_{\sqrt{\varepsilon},t}$ are polynomial-geometric sequences and are summable by Lemma \ref{poly_lemma}.
	\par By the summability of $b_{e,t}$, $b_{\varepsilon,t}$ and $b_{\sqrt{\varepsilon},t}$, there exist some scalars $\textbf{A}$ and $\textbf{B}$ such that 
	\begin{align}\label{Atb}
	A_T=\sum_{t=0}^{T-1}(\gamma n b_{e,t}+n\gamma^2C_0+\sqrt{2\gamma}b_{\sqrt{\varepsilon},t})\leq \textbf{A} <\infty
	\end{align}
	{where we use the step-size range $\gamma\leq 1/\sqrt{T}$.} And 
	\begin{align}\label{Btb}
	B_T= &T \textbf{D}+12\gamma^3 L n^3 \sum_{t=0}^{T-1} b_{e,t}^2 +2\gamma  (1+2\gamma nL)\sum_{t=0}^{T-1}b_{\varepsilon,t},\notag\\
	\leq & T \textbf{D}+\textbf{B},
	\end{align}
	where $\textbf{D}\triangleq \gamma^3(9 L n G_f^2+3L {n^2}\sigma_*^2+6 n L G_{\phi}^2+12 n^3 \gamma^2  C_0^2)$.
	Then, $\mathbb{E}[\|\bar{x}_T-x_*\|]$ satisfies
	\begin{align*}
	\mathbb{E}[\|\bar{x}_T-x_*\|] \leq \textbf{A}+(\|\bar{x}_0-x_*\|^2+(T\textbf{D}+\textbf{B})+\textbf{A}^2)^{\frac{1}{2}}.
	\end{align*}
	Since $\{A_t\}$ and $\{B_t\}$ are increasing sequences, we have for $t\leq T$,
	\begin{align}\label{xtsub}
	\mathbb{E}[\|\bar{x}_t-x_*\|]\leq& A_T+(\|\bar{x}_0-x_*\|^2+B_T+A_T^2)^{\frac{1}{2}}\notag\\
	\leq & 2A_T+\|\bar{x}_0-x_*\|+\sqrt{B_T}\notag\\
	\leq &2\textbf{A}+\|\bar{x}_0-x_*\|+\sqrt{T\textbf{D}+\textbf{B}}.
	\end{align}
	Now, we can bound the right-hand side of \eqref{fsubfstar} with \eqref{xtsub},
	
	\begin{align*}
	&2 \gamma \sum_{t=1}^{T} \mathbb{E}[F(\bar{x}_{t})-F_*]\notag\\
	\leq &\|\bar{x}_{0}-x_*\|^2+ B_T+ 2A_T( 2A_T+\|\bar{x}_0-x_*\|+\sqrt{B_T})\\
	\leq &\|\bar{x}_{0}\!-\!x_*\|^2\!+\! T\textbf{D}\!+\!\textbf{B}\!+\!2\textbf{A}(2\textbf{A}\!+\!\|\bar{x}_0-x_*\|\!+\!\sqrt{T\textbf{D}\!+\!\textbf{B}}).
	\end{align*}
	By dividing both sides by $2\gamma T$, we get
	\begin{align*}
		&\frac{1}{T}\sum_{t=1}^{T} \mathbb{E}[F(\bar{x}_{t})-F_*]\notag\\
		\leq &\frac{\|\bar{x}_{0}\!-\!x_*\|^2}{2\gamma T}\!+\! \frac{\textbf{D}}{2\gamma}\!+\!\frac{\textbf{B}}{2\gamma T}\!+\! \frac{\textbf{A}}{\gamma T}\!( 2\textbf{A}\!+\!\|\bar{x}_0\!-\!x_*\|\!+\!\sqrt{\!T\textbf{D}\!+\!\textbf{B}}).
	\end{align*}
Finally, by the convexity of $F$, the average iterate $\hat{x}_T=\frac{1}{T}\sum_{t=1}^T \bar{x}_t$ satisfies
\begin{align}\label{fsubexpec}
&\mathbb{E}[F(\hat{x}_T)-F_*]\notag\\
\leq& \frac{1}{T}\sum_{t=1}^{T} \mathbb{E}[F(\bar{x}_{t})-F_*]\notag\\
\leq &\frac{\|\bar{x}_{0}\!-\!x_*\|^2}{2\gamma T}\!+\! \frac{\textbf{D}}{2\gamma}\!+\!\frac{\textbf{B}}{2\gamma T}\!+\! \frac{\textbf{A}}{\gamma T}( 2\textbf{A}\!+\!\|\bar{x}_0\!-\!x_*\|\!+\!\sqrt{\!T\textbf{D}\!+\!\textbf{B}})\notag\\
\leq &\frac{\|\bar{x}_{0}-x_*\|^2}{2\gamma T}+\frac{\textbf{A}}{\gamma T}( 2\textbf{A}+\|\bar{x}_0-x_*\|)
+\frac{\textbf{A}\sqrt{\textbf{B}}}{\gamma T}+ \frac{\textbf{B}}{2\gamma T}\notag\\
&+\frac{\textbf{A}\sqrt{\textbf{D}}}{\gamma \sqrt{T}}+\frac{\textbf{D}}{2\gamma},\notag\\
=&\mathcal{O}(\frac{1}{T}+\frac{1}{T^{1/2}})
\end{align}
{where $\gamma\leq T^{-1/2}$ and the last equality holds due to the fact that $\sqrt{\mathbf{D}}/{\gamma}=\mathcal{O}(T^{-1/4})$, and $\mathbf{D}/{\gamma}=\mathcal{O}(T^{-1})$.}
\end{proof}
\section{Simulation}\label{simulation}
In this section, we apply the proposed algorithm DPG-RR to optimize the black-box binary classification problem \cite{log_simu}, which is to find the optimal predictor $x\in \mathbb{R}^n$ by solving 
\begin{align}\label{black_box}
	{\rm min}_{x\in \mathbb{R}^d} F(x),\ F(x)=\frac{1}{m} \sum_{j=1}^m f_j(x) +\lambda_1 \|x\|_1,
\end{align}
where $f_j(x)\triangleq \sum_{i=1}^n \ln (1+\exp(-l_{j,i}\langle a_{j,i},x\rangle))$, $a_{j,i}\in \mathbb{R}^d$ is the feature vector of the $i$th local sample of agent $j$, $l_{j,i} \in \{-1,1\}$ is the classification value of the $i$th local sample of agent $j$ and $\{a_{j,i},l_{j,i}\}_{i=1}^n$ denotes the set of local training samples of agent $j$. 
 In this experiment, we use the publicly available real datasets a9a and w8a \footnote{a9a and w8a are available in the website www.csie.ntu.edu.tw/~cjlin/libsvmtools/datasets/.}. In addition, in our setting, there are $m=10$ agents and $\lambda_1=5\times 10^{-4}$. All comparative algorithms are initialized with a same value. Experiment codes and the adjacent matrices of ten-agent time-varying networks are provided at \url{https://github.com/managerjiang/Dis-Prox-RR}.
 \begin{figure}
 	\centering
 	\subfigure[Consensus trajectories]{
 		\includegraphics[width=6cm]{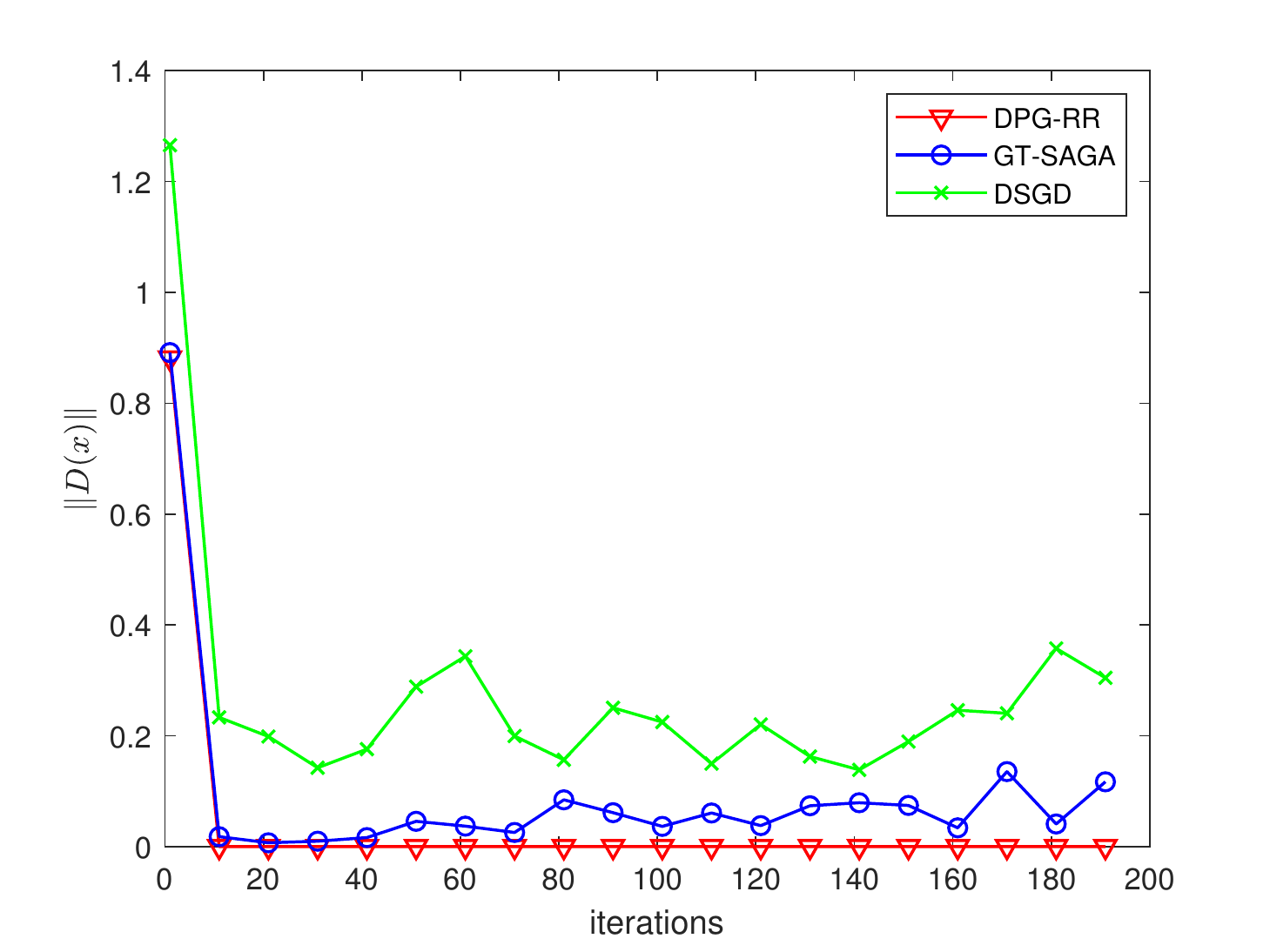}
 		\label{a9a_con_sto}	
 	}
 
 	\subfigure[ Cost evolutions]{
 		\includegraphics[width=7.5cm,height=5.2cm]{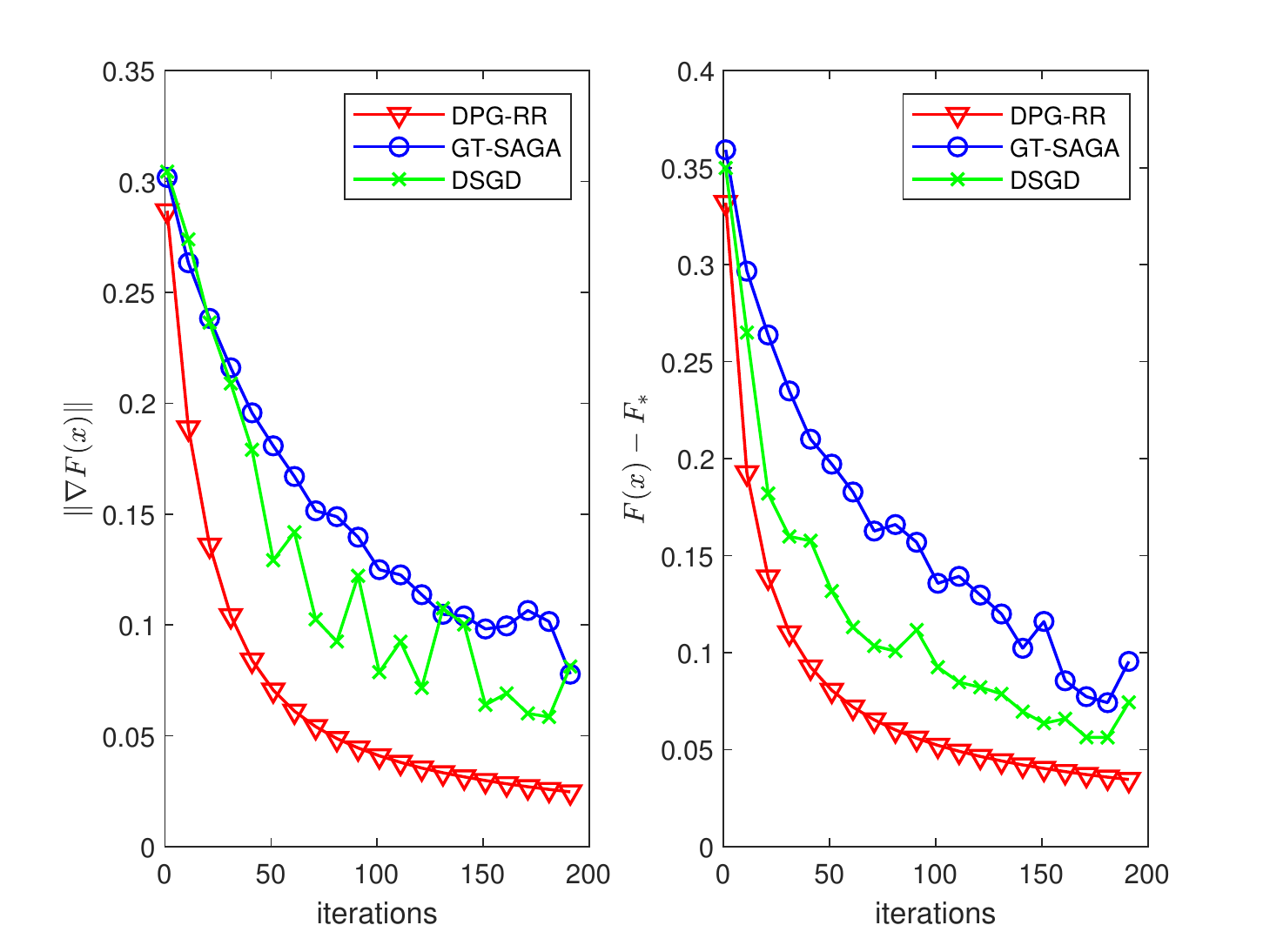}
 		\label{a9a_f_sto}	
 	}
 	\caption{ Different algorithms for a9a dataset over time-invariant graphs}
 	\label{a9a_fig_sto}
 \end{figure}
 \begin{figure}
 	\centering
 	\subfigure[Consensus trajectories]{
 		\includegraphics[width=6cm]{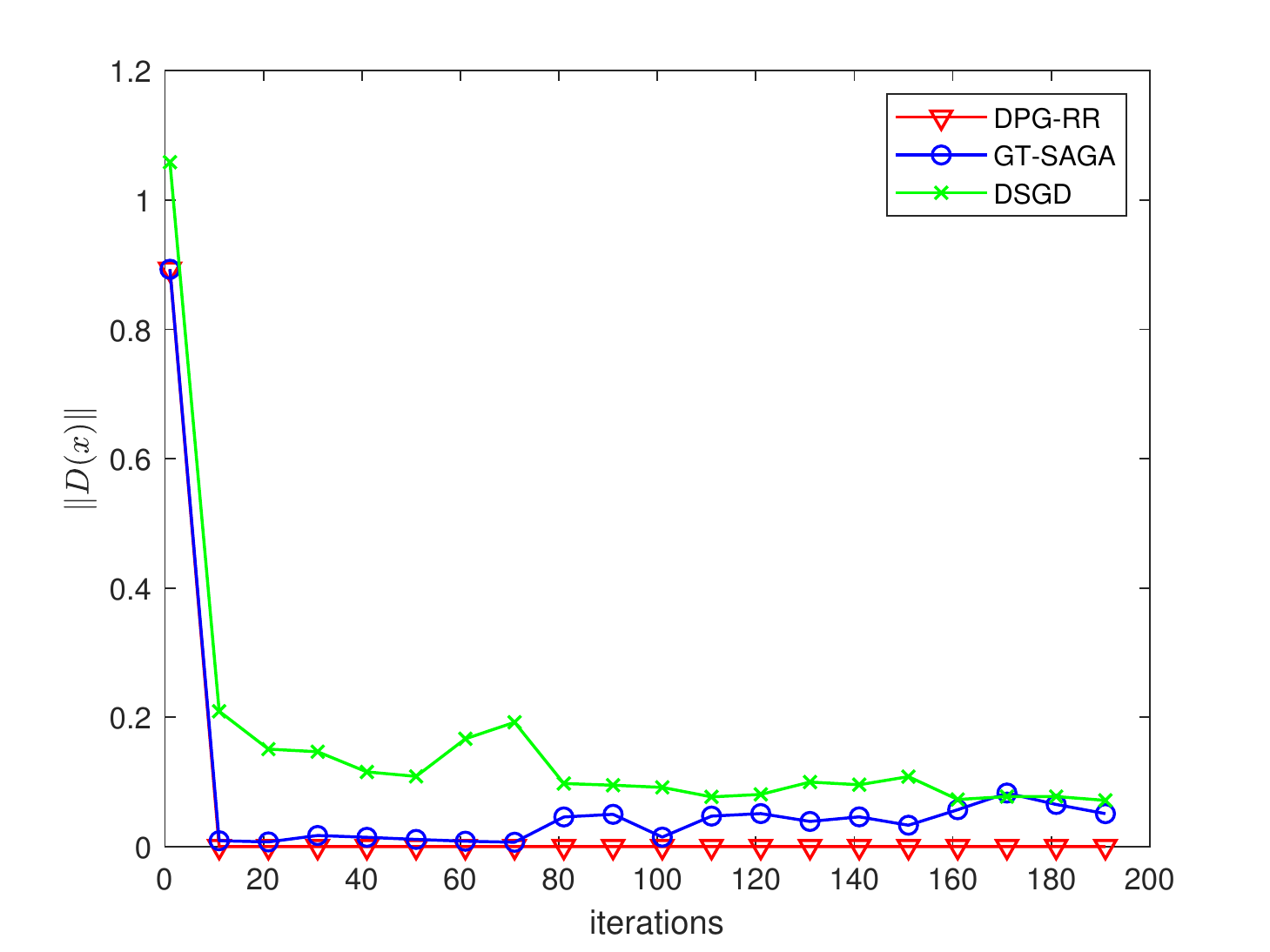}
 		\label{w8a_con_sto}	
 	}
 	\subfigure[ Cost evolutions]{
 		\includegraphics[width=7.5cm,height=5.2cm]{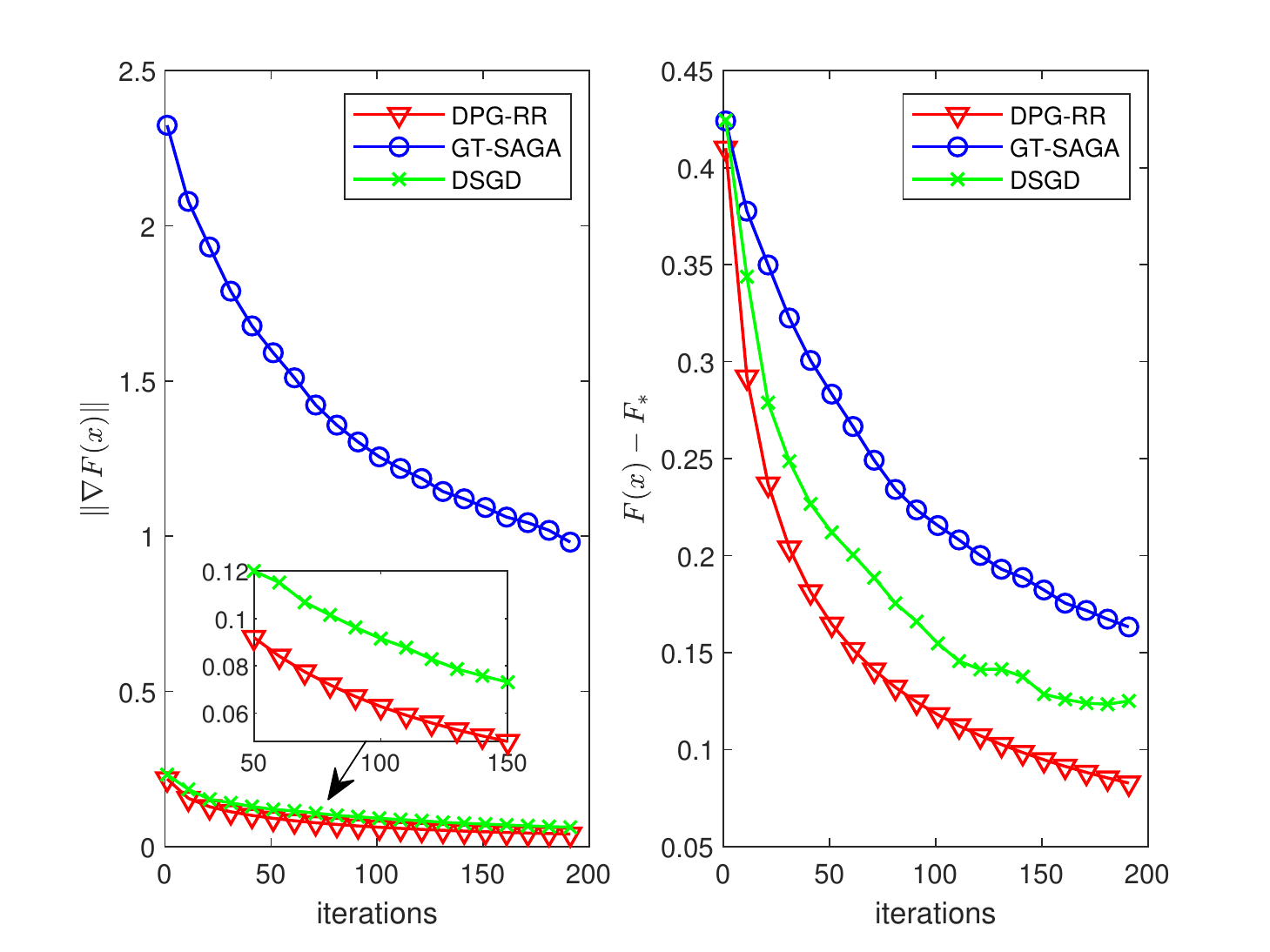}
 		\label{w8a_f_sto}	
 	}
 	\caption{ Different algorithms for w8a dataset over time-invariant graphs}
 	\label{w8a_fig_sto}
 \end{figure}
 {
  \par (1) Time-invariant networks: For comparison, we apply some existing distributed stochastic algorithms including GT-SAGA in \cite{2020TSP} and DSGD in \cite{DPSG}, and the proposed stochastic DPG-RR to solve \eqref{black_box} over the ten-agent  connected network. Since the cost function in \eqref{black_box} is non-smooth, we replace the gradients in GT-SAGA and DSGD with subgradients. The simulation results for a9a and w8a datasets are shown in Figs. \ref{a9a_fig_sto} and \ref{w8a_fig_sto}.  Fig. \ref{a9a_f_sto} and Fig. \ref{w8a_f_sto} show that the trajectories of objective function $F(x)$ generated by all algorithms converge quickly to the optimum $F_*$. To measure the consensus performance, we define one quantity $D(x)$\footnote{The $D(x)$ is an expansion of $\hat{x}^\top L \hat{x}$, where $L\triangleq \mathcal{L}\otimes I_d\in \mathbb{R}^{md\times md}$, $\mathcal{L}$ is the Laplacian matrix of the undirected multi-agent network and $\hat{x}=[x_1,\cdots,x_m]^\top\in \mathbb{R}^{md}$. In addition, the Laplacian matrix of an undirected multi-agent network is positive semi-definite. By the fact that $0$ is the single eigenvalue corresponding to eigenvector $1_{md}$ of $L$, $L\hat{x} = 0_{md}$ if and only if $x_i = x_j$ for all $i,j \in \{1,\cdots,m\}$.} as $D(x)=\sum_{i=1}^{m}x_i'\sum_{j=1}^m a_{ij}(x_i-x_j)$, where $a_{ij}$ is $(i,j)$th element of one doubly-stochastic adjacent matrix. If all local variable estimates achieve consensus, then, $D(x)=0$. Fig. \ref{a9a_con_sto} and Fig. \ref{w8a_con_sto} represent that local variables generated by all algorithms achieve consensus. In addition, it is seen that the proposed algorithm DPG-RR owns a better consensus performance than GT-SAGA and DSGD, which verify the effectiveness of the multi-step consensus mapping in DPG-RR.
  
\par Although all comparative algorithms have comparable convergence rate in practice, \cite{2020TSP} and \cite{DPSG} do not provide the theoretical convergence analysis of GT-SAGA and DSGD for non-smooth optimization. In addition, these algorithms are not applicable for time-varying multi-agent networks, hence, we further compare our proposed algorithm with some recent distributed deterministic algorithms for non-smooth optimization over time-varying networks.
}
\par (2) Time-varying networks: We compare DPG-RR with the distributed deterministic proximal algorithm Prox-G in \cite{proximal-convex}, the algorithm NEXT in \cite{NEXT2016} and the distributed subgradient method (DGM) in \cite{DGM-nedic}  to solve \eqref{black_box} over the time-varying graphs. DPG-RR and Prox-G adopt a same constant step-size, NEXT and DGM have diminishing step-sizes.
\begin{figure}
	\centering
	\includegraphics[width=8cm]{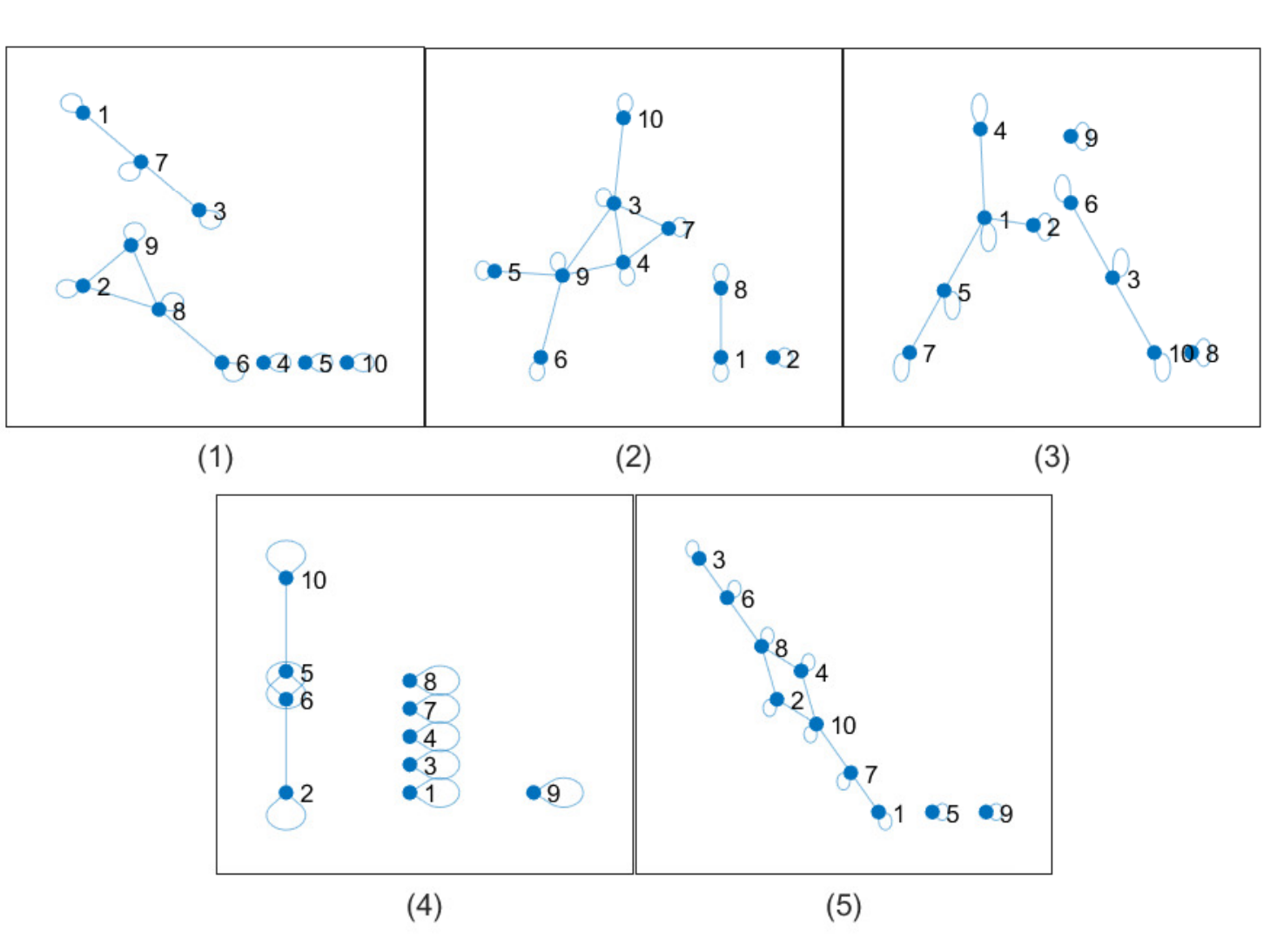}
	\caption{Ten-agent periodically time-varying networks}
	\label{topo_fig}	
\end{figure}
All distributed algorithms are applied over {periodically time-varying} ten-agent networks satisfying Assumption \ref{net_assum} to solve the problem \eqref{black_box}, which are shown in Fig. \ref{topo_fig}. 

\begin{figure}
	\centering
	\subfigure[Consensus trajectories]{
		\includegraphics[width=6cm]{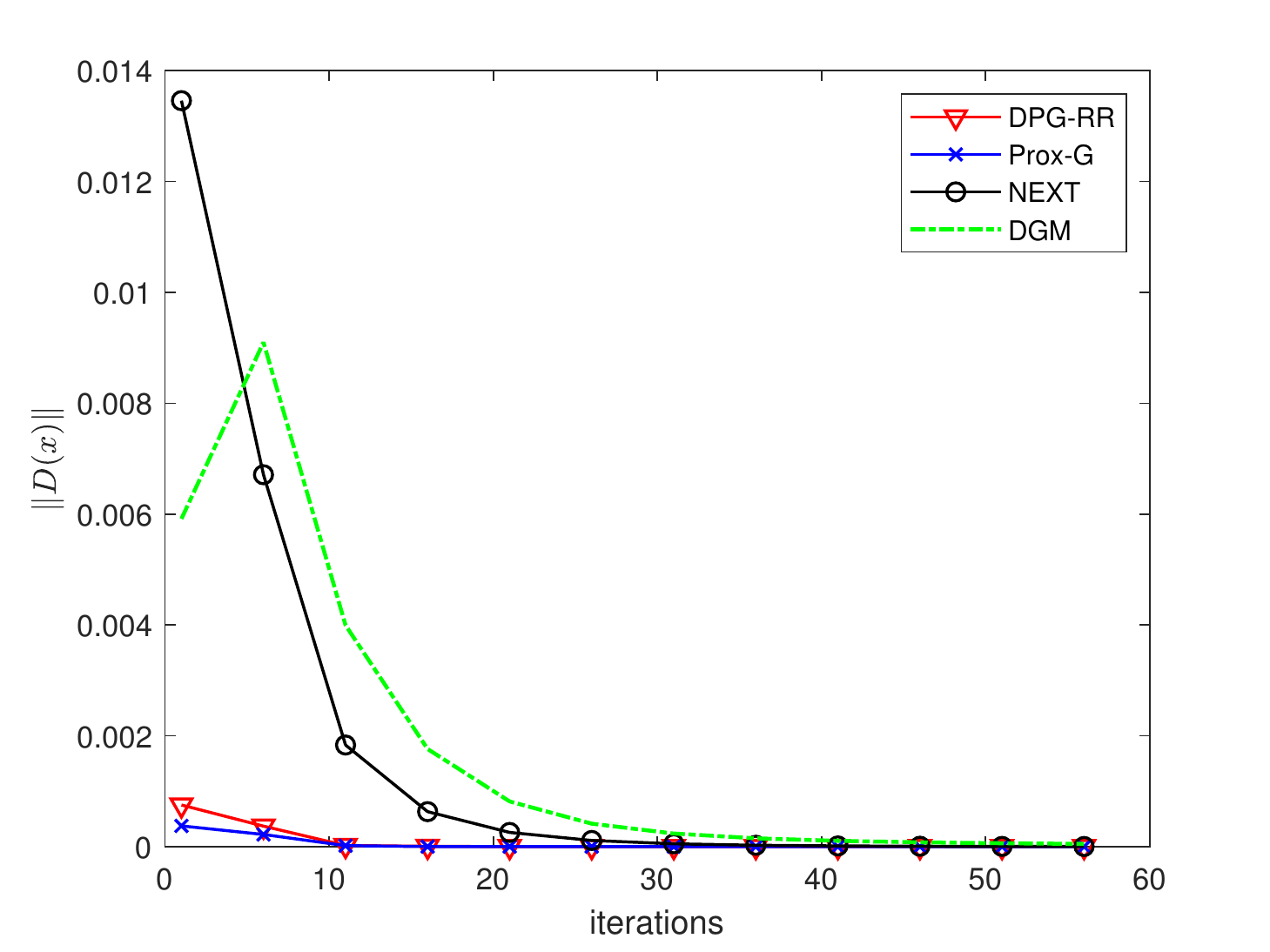}
		\label{a9a_con}	
	}
	\subfigure[ Cost evolutions]{
		\includegraphics[width=7.5cm]{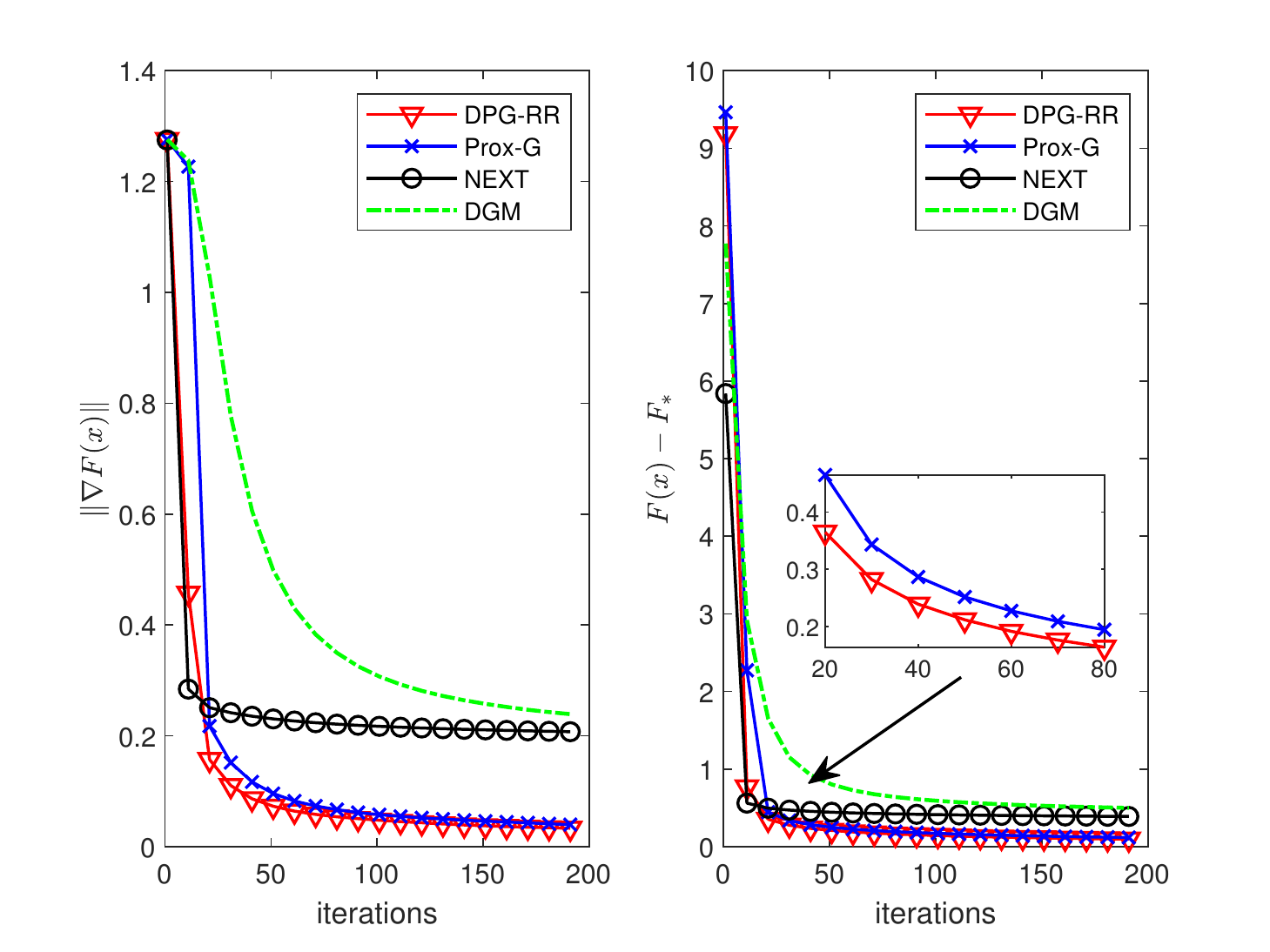}
		\label{a9a_f}	
	}
	\caption{ Different algorithms for a9a dataset over time-varying graphs}
	\label{a9a_fig}
\end{figure}

\begin{figure}
	\centering
	\subfigure[Consensus trajectories]{
	\includegraphics[width=6cm]{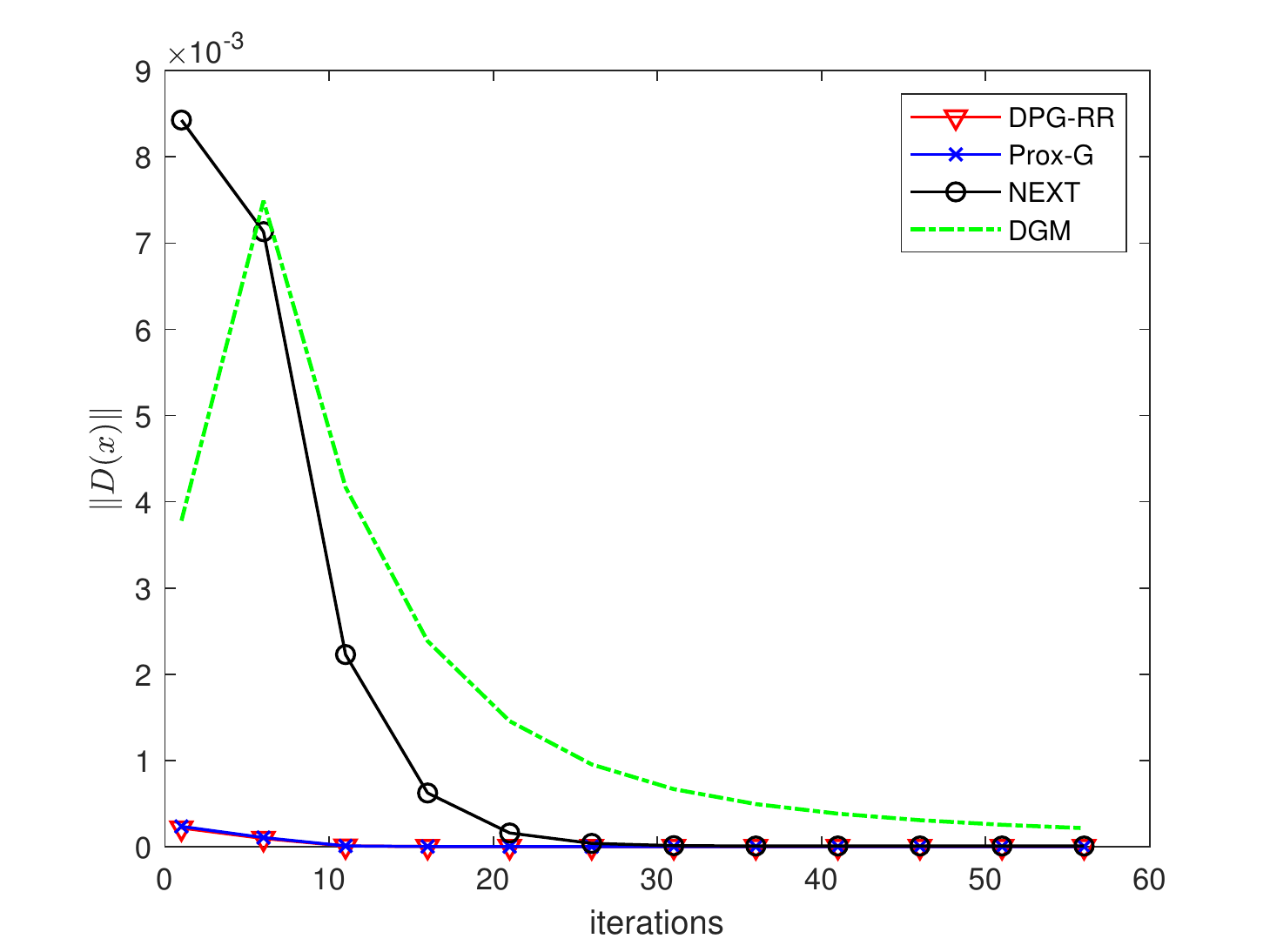}
	\label{w8a_con}	
	}
	\subfigure[ Cost evolutions]{
	\includegraphics[width=7.5cm]{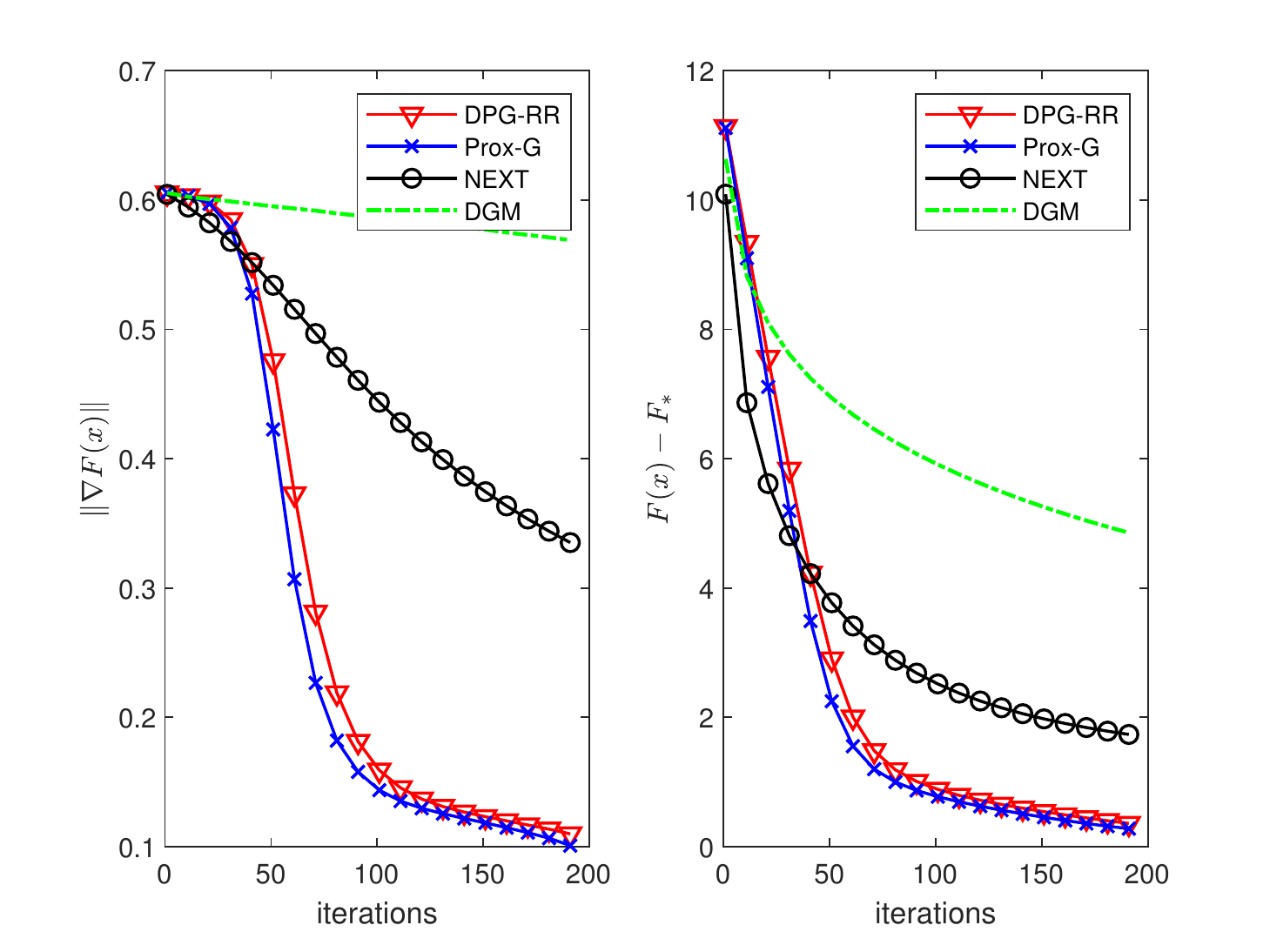}
	\label{w8a_f}	
	}
	\caption{ Different algorithms for w8a dataset over time-varying graphs}
	\label{w8a_fig}
\end{figure}
\begin{table}
	\caption{Running time of different algorithms}\label{run_time_table}
	\centering
	\begin{tabular}{|c|cc|}
		\hline
		& \multicolumn{2}{c|}{ time(s)} \\ \cline{2-3} 
		\multirow{-2}{*}{Algorithms} & \multicolumn{1}{c|}{a9a}                       & w8a                        \\ \hline
		DPG-RR                       & \multicolumn{1}{c|}{25.21}                     & 140.89                     \\ \hline
		Prox-G                       & \multicolumn{1}{c|}{55.50}                     & 161.44                     \\ \hline
		NEXT                         & \multicolumn{1}{c|}{80.28}                     & 264.78                     \\ \hline
		DGM                          & \multicolumn{1}{c|}{30.91}                     & 120.51                     \\ \hline
	\end{tabular}
\end{table} 
\par The simulation results for a9a and w8a datasets are shown in Fig. \ref{a9a_fig} and Fig. \ref{w8a_fig}, respectively. 
It is seen from Fig. \ref{a9a_con} and Fig. \ref{w8a_con} that the trajectories generated by DPG-RR, Prox-G, NEXT and DGM all converge to zero, implying that local variable estimates generated by all algorithms achieve consensus. In addition, the DPG-RR shows a better consensus numerical performance than NEXT and DGM algorithms. 
From Fig. \ref{a9a_f} and Fig. \ref{w8a_f}, we can see that DPG-RR converges faster than NEXT and DGM over both a9a and w8a datasets, which verify that DPG-RR is superior to algorithms with diminishing step-sizes. 
We observe that DPG-RR has a comparable convergence performance of Prox-G. While, due to the stochastic technique, DPG-RR avoids the computation of full local gradients in Prox-G. 
 The running time of different algorithms  are shown in Table \ref{run_time_table}. It is observed that the proposed stochastic DPG-RR generally costs less running time than the deterministic algorithms Prox-G and NEXT.
	\begin{figure}
		\centering
		\includegraphics[width=7.5cm,height=5.2cm]{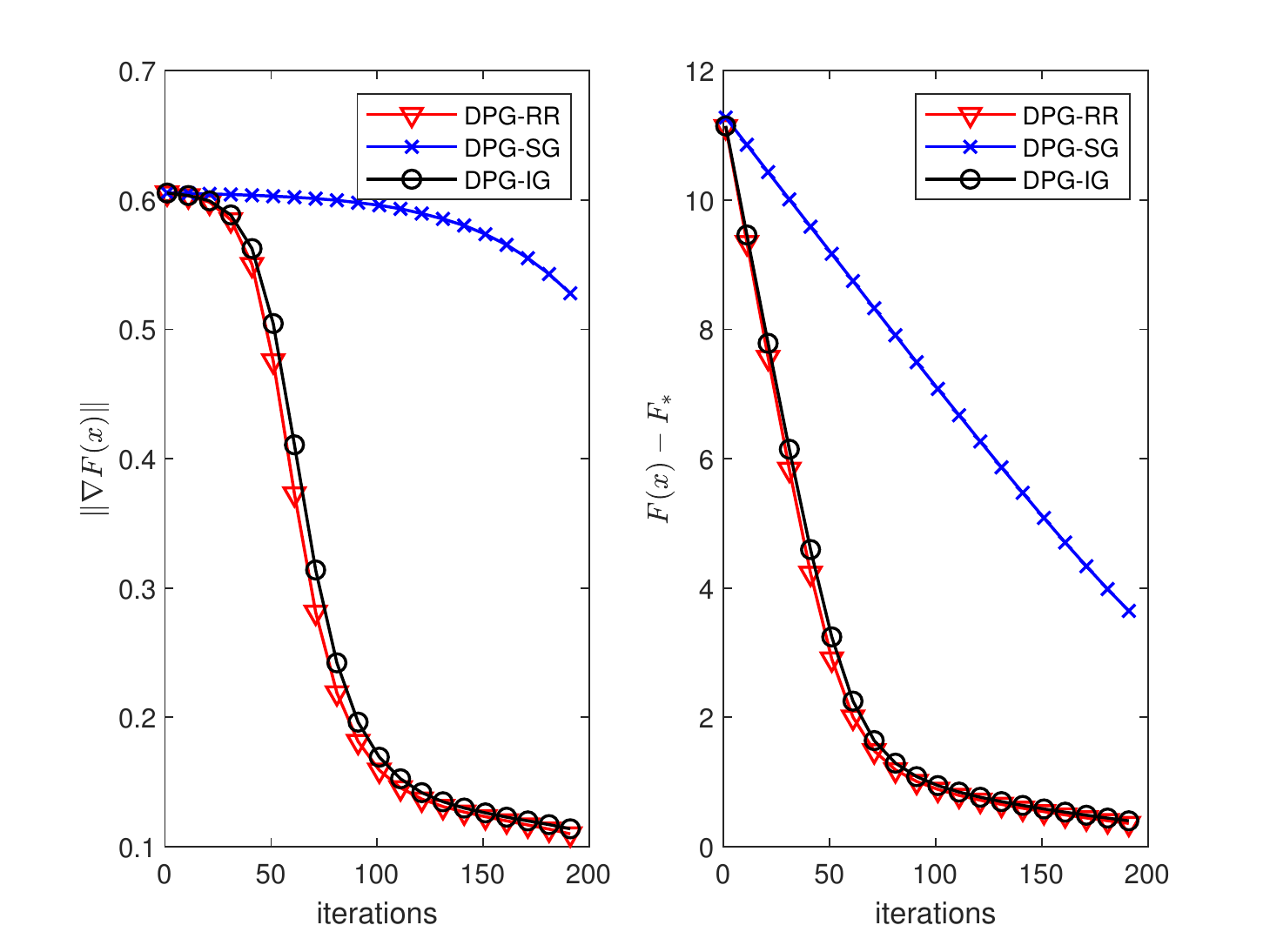}
		\caption{ Algorithms with different sampling procedures for w8a dataset}
		\label{w8a_samp_fig}
	\end{figure} 
\par To show the advantages of random reshuffling sampling procedure, we replace the random reshuffling step of DPG-RR with a stochastic sampling procedure, yielding a DPG-SG algorithm. Similarly, we replace the random reshuffling step of  DPG-RR with a deterministic incremental sampling procedure, yielding a DPG-IG algorithm. All algorithms take the same constant step-size ($0.3$) and the simulation results are shown in Fig. \ref{w8a_samp_fig}. Compared with fully stochastic sampling procedure (DPG-SG), the proposed DPG-RR shows a faster convergence. In addition, the DPG-IG algorithm shows a comparable convergence rate with DPG-RR.  These numerical results are consistent with the theoretical findings discussed in Remark \ref{sample_remark}. 


\section{Conclusion}\label{conclusion}
Making use of random reshuffling, this paper has developed one distributed stochastic proximal-gradient algorithm for solving large-scale non-smooth convex optimization over time-varying multi-agent graphs. The proposed algorithm operates only one single proximal evaluation at each epoch, owning significant advantages in scenarios where the proximal mapping is computationally expensive. In addition, the proposed algorithm owns constant step-sizes, overcoming  the degraded performance of most stochastic algorithms with diminishing step-sizes. One future research direction is to improve the convergence rate of DPG-RR by some accelerated methods, such as Nesterov's acceleration technique or variance reduction technique.  Another one is to extend the proposed algorithm for distributed non-smooth non-convex optimization, which is applicable for deep neural network models.

\section{Appendix}
\subsection{Some vital Lemmas}
At first, the following lemma characterizes the $\varepsilon_t$-subdifferential of non-smooth function $\phi$ at $\bar{x}_t$, $\partial_{\varepsilon_t}\phi(\bar{x}_t)$. 
\begin{lemma}\cite[Lemma 2]{ Schmidt2012}\label{pk_lemma}
	If $\bar{x}_t$ is an $\varepsilon_{t}$-optimal solution to \eqref{prox_ope} in the sense of \eqref{inprox_def} with $y=\bar{x}_{t-1}-\gamma ( \bar{g}_{t-1}+e_t)$, then there exists $p_t\in \mathbb{R}^d$ such that
	\begin{align}\label{pt_range}
	\|p_t\|\leq \sqrt{2 \gamma \varepsilon_t}
	\end{align}
	and
	$\frac{1}{ \gamma}(\bar{x}_{t-1}-\bar{x}_{t}-\gamma \bar{g}_{t-1}-\gamma e_{t}-p_{t})\in \partial_{\varepsilon_{t}}\phi(\bar{x}_{t}).$
\end{lemma}
\par The next lemma shows that polynomial-geometric sequences are summable, which is vital for the analysis of error sequences introduced by transformation.
\begin{lemma}\label{poly_lemma}\cite[Proposition 3]{proximal-convex}
	Let $\zeta \in (0,1)$, and let
	$$P_{k,N}=\{c_N k^N+\cdots+c_1 k+c_0 | c_j\in \mathbb{R}, j=0,\cdots,N\}$$
	denote the set of all $N$-th order polynomials of $k$, where $N \in \mathbb{N}$. Then for every polynomial $p_{k}\in P_{k,N}$,
	$$\sum_{k=1}^{\infty} p_{k} \zeta^k<\infty.$$
\end{lemma}
\par The result of this Lemma for $P_{k,N}=\{k^N\}$ will be particularly useful for the analysis in the following. Hence, we define
\begin{align}\label{kn_poly}
S_{N}^{\zeta}\triangleq \sum_{k=1}^{\infty} k^N \zeta^k<\infty.
\end{align}
\par The following lemma characterizes the variance of sampling without replacement, which is a key ingredient in our convergence result of DPG-RR.
\begin{lemma}\cite[Lemma 1]{NEURIPS2020_c8cc6e90}\label{xavelem}
	Let $X_1,\cdots,X_n \in \mathbb{R}^d$ be fixed vectors, $\bar{X}\triangleq \frac{1}{n}\sum_{i=1}^n X_i$ be their average and $\sigma^2\triangleq \frac{1}{n}\sum_{i=1}^n \|X_i-\bar{X}\|^2$ be the population variance. Fix any $k\in \{1,\cdots,n\}$, let $X_{\pi_1},\cdots, X_{\pi_k}$ be sampled uniformly without replacement from $\{X_1,\cdots, X_n\}$ and $\bar{X}_{\pi}$ be their average. Then, the sample average and variance are given by 
	\begin{align*}
	\mathbb{E}[\bar{X}_{\pi}]=\bar{X}, \quad \mathbb{E}[\|\bar{X}_{\pi}-\bar{X}\|^2]=\frac{n-k}{k(n-1)}\sigma^2.
	\end{align*}
\end{lemma}
Finally, we provide a lemma on the boundness of one special nonnegative sequence, which is vital for the convergence analysis.
\begin{lemma}{\cite[Lemma 1] {Schmidt2012}}\label{usseq}
	Assume that the nonnegative sequence $\{u_T\}$ satisfies the following recursion for all $k\geq 1$
	\begin{align}
	u_T^2\leq S_T+\sum_{t=1}^T \lambda_t u_t,
	\end{align}
	with $\{S_T\}$ an increasing sequence, $S_0\geq u_0^2$ and $\lambda_t\geq 0$ for all $t$. Then, for all $T\geq 1$,
	\begin{align}
	u_T\leq \frac{1}{2} \sum_{t=1}^T \lambda_t +\Big(S_T+\big(\frac{1}{2}\sum_{t=1}^T \lambda_t\big)^2\Big)^{\frac{1}{2}}.
	\end{align}
\end{lemma}
\subsection{Proof of Proposition \ref{aver_up}}\label{center_pro}
\begin{proof}
	By taking the average of \eqref{x_ji_up} over $m$ agents,
	\begin{align}\label{barx_up}
	\bar{x}_t^{i+1}=&\bar{x}_t^{i}-\gamma (\frac{1}{m}\sum_{j=1}^m \nabla f_{j,\pi_t^i}(\bar{x}_{t}^i) +e_t^{i})\notag\\
	=&\bar{x}_t^{i}-\gamma (\nabla \mathbf{f}_{\pi_t^i}(\bar{x}_{t}^i) +e_t^{i}),
	\end{align}
	where $e_t^i=\frac{1}{m}\sum_{j=1}^m (\nabla  f_{j,\pi_t^i}(x_{j,t}^i)-\nabla f_{j,\pi_t^i}(\bar{x}_{t}^i))$ and $ \mathbf{f}_{\pi_t^i}(x)=\frac{1}{m}\sum_{j=1}^m  f_{j,\pi_t^i}(x)$. 
	Because of the Lipschitz-continuity of the gradient of $f_{j,\pi_t^i}(x)$, $\|e_t^i\|$ satisfies
	\begin{align}
	\|e_t^i\|\leq \frac{L}{m} \sum_{j=1}^m \|x_{j,t}^i-\bar{x}_t^i\|.
	\end{align}
	\par Recall that
	$z_{t+1}\triangleq {\rm prox}_{\gamma, \phi}(\bar{v}_t)={\rm argmin}_x\{\phi(x)+\frac{1}{2\gamma}\|x-\bar{v}_t\|^2\}$,
	which is the result of the exact centralized proximal step. Then, we relate $z_{t+1}$ and $\bar{x}_{t+1}$ by formulating the latter as an inexact proximal step with error $\varepsilon_{t+1}$. A simple algebraic expansion gives 
	\begin{align*}
	&\phi(\bar{x}_{t+1})+\frac{1}{2\gamma}\|\bar{x}_{t+1}-\bar{v}_t\|^2\\
	\leq & \phi(z_{t+1})+G_{\phi} \|\bar{x}_{t+1}-z_{t+1}\|+\frac{1}{2\gamma}\Big(\|z_{t+1}-\bar{v}_t\|^2\\
	&+2\langle z_{t+1}-\bar{v}_t,\bar{x}_{t+1}-z_{t+1}\rangle+\|\bar{x}_{t+1}-z_{t+1}\|^2\Big)\\
	\leq &\min_{z\in\mathbb{R}^d}\{\phi(z)+\frac{1}{2\gamma}\|z-\bar{v}_t\|^2\}+\frac{1}{2\gamma} \|\bar{x}_{t+1}-z_{t+1}\|^2\\
	&+\|\bar{x}_{t+1}-z_{t+1}\|(G_{\phi}+\frac{1}{\gamma}\|z_{t+1}-\bar{v}_t\| ),
	\end{align*}
	{where we have used the fact that $\phi(z_{t+1}) + \frac{1}{2\gamma}\Vert z_{t+1}-\bar{v}_t \Vert^2 = \min_{z}\{\phi(z)+\frac{1}{2\gamma}\Vert z-\bar{v}_t \Vert^2 \}$ and $\langle z_{t+1}-\bar{v}_t,\bar{x}_{t+1}-z_{t+1}\rangle\leq \|z_{t+1}-\bar{v}_t\|\|\bar{x}_{t+1}-z_{t+1}\|$ in the last inequality.}
	\par Therefore, we can write
	\begin{align}
	\bar{x}_{t+1}\in {\rm prox}_{\gamma, \phi}^{\varepsilon_{t+1}}(\bar{v}_t),
	\end{align}
	where $\varepsilon_{t+1}=\|\bar{x}_{t+1}-z_{t+1}\|(G_{\phi}+\frac{1}{\gamma}\|z_{t+1}-\bar{v}_t\| )+\frac{1}{2\gamma} \|\bar{x}_{t+1}-z_{t+1}\|^2.$
	\par By definition, $z_{t+1}$ satisfies $\frac{1}{\gamma}(\bar{v}_t-z_{t+1})\in \partial \phi(z_{t+1})$, and therefore its norm is bounded by $G_{\phi}$. As a result,
	\begin{align}\label{varep1}
	\varepsilon_{t+1}\leq 2G_{\phi}\|\bar{x}_{t+1}-z_{t+1}\|+\frac{1}{2\gamma}\|\bar{x}_{t+1}-z_{t+1}\|^2.
	\end{align}
	Combined with the nonexpansiveness of the proximal operator,
	\begin{align}\label{xz1}
	\|\bar{x}_{t+1}-z_{t+1}\|\leq& \frac{1}{m} \sum_{j=1}^m\|{\rm prox}_{\gamma, \phi}(v_{j,t})-{\rm prox}_{\gamma, \phi}(\bar{v}_t)\|\notag\\
	\leq &\frac{1}{m}\sum_{j=1}^m \|v_{j,t} -\bar{v}_t\|.
	\end{align}
	Finally, substituting \eqref{xz1} to \eqref{varep1}, we obtain \eqref{norm_vare}.
\end{proof}

\subsection{Proof of Proposition \ref{e_sum_pro}}\label{summ_proof}
\par {We first define two useful quantities $\Gamma\triangleq 2\frac{1+\eta^{-(m-1)B}}{1-\eta^{(m-1)B}}$ and $\Xi\triangleq (1-\eta^{(m-1)B})^{\frac{1}{(m-1)B}}$.} Then, we provide some properties of the local variables generated by the proposed algorithm in the following Lemma.
\begin{lemma}\label{ite_seq}
	Under Assumptions \ref{f_assump} and \ref{net_assum}, for each iteration $t\geq 2$,
		\begin{align}\label{xn}
		\sum_{j=1}^m \|x_{j,t}^n\|\leq \sum_{j=1}^m\|x_{j,t-1}^n\|+m\gamma (G_{\phi}+n G_{f}),
		\end{align}
		\begin{align}\label{ite_b}
		\sum_{j=1}^m \|x_{j,t+1}-x_{j,t}\|
		\!\leq \! 2m\Gamma \Xi^{t-1} \sum_{l=1}^m\|x_{l,t-1}^n\|\!+\!m\gamma (G_{\phi}\!+\!nG_{f}),
		\end{align}
		\begin{align}\label{ite_c}
		\sum_{j=1}^m\|x_{j,t}-\bar{x}_{t} \|
		\leq  2 m \Gamma \Xi^t \sum_{l=1}^m\|x_{l,t-1}^n\|.
		\end{align} 
\end{lemma}
\begin{proof}
	By \eqref{xj_up}, there exists $z_{j,t+1}\in \partial \phi(x_{j,t+1})$ such that
	\begin{align}\label{z_de}
		x_{j,t+1}=v_{j,t}-\gamma z_{j,t+1}.
	\end{align}
	Since function $\phi$ has bounded subgradients,
	\begin{align}\label{xsubv}
		\|x_{j,t+1}-v_{j,t}\|\leq \gamma G_{\phi}.
	\end{align}
	\par (a) By the algorithm \ref{pg_rr_algo}, we have
	\begin{align*}
		x_{j,t}^{n}=x_{j,t}^0-\gamma \sum_{i=0}^{n-1} \nabla f_{j,\pi_t^{i}}(x_{j,t}^i).
	\end{align*}
	Taking norm of the above equality and summing over $j$,
	{
	\begin{align}\label{xjtn}
		\sum_{j=1}^m \|x_{j,t}^{n}\|=&\sum_{j=1}^m \|x_{j,t}^0-\gamma \sum_{i=0}^{n-1} \nabla f_{j,\pi_t^{i}}(x_{j,t}^i)\|\notag\\
		\leq & \sum_{j=1}^m \|x_{j,t}\|+\gamma m n G_f.
	\end{align}
	}
	By \eqref{xsubv}, $\|x_{j,t+1}\|-\|v_{j,t}\|\leq \gamma G_{\phi}$ holds. Since $v_{j,t}$ is a convex combination of $\{x_{l,t}^n\}_{l=1}^m$ and $\sum_{l=1}^m \lambda_{jl,t}=1$, 
	\begin{align}\label{vsumx}
		\sum_{j=1}^m\|v_{j,t}\|\leq \sum_{j=1}^m \|x_{j,t}^n\|.
	\end{align}
	Then, substituting the fact that $\|x_{j,t+1}\|-\|v_{j,t}\|\leq \gamma G_{\phi}$ and \eqref{vsumx} to \eqref{xjtn}, we obtain
	\begin{align*}
		\sum_{j=1}^m \|x_{j,t}^n\|\leq \sum_{j=1}^m\|x_{j,t-1}^n\|+m\gamma (G_{\phi}+nG_{f}).
	\end{align*}
	\par (b) By \eqref{z_de} and the proposed algorithm, 
	\begin{align}\label{xjtpuls}
		x_{j,t+1}=&v_{j,t}-\gamma z_{j,t+1}\notag\\
		\overset{\eqref{v_ji_up}}{=}&\sum_{l=1}^m\lambda_{jl,t}x_{l,t}^n-\gamma z_{j,t+1}\notag\\
		\overset{\eqref{x_ji_up}}{=}&\sum_{l=1}^m\lambda_{jl,t}\Big(x_{l,t}^0-\gamma \sum_{i=0}^{n-1}\nabla f_{l,\pi_t^i}(x_{l,t}^i)\Big)-\gamma z_{j,t+1}.
	\end{align}
	Then, subtracting $x_{j,t}$ from both sides of \eqref{xjtpuls} and taking norm, we obtain
	\begin{align}\label{xsbus}
		&\sum_{j=1}^m \|x_{j,t+1}-x_{j,t}\|\notag\\
		\leq &\sum_{j=1}^m\sum_{l=1}^m\lambda_{jl,t}\|x_{l,t}^0-x_{j,t}\|+m\gamma (G_{\phi}+nG_f)\notag\\
		= &\sum_{j=1}^m\sum_{l=1}^m\lambda_{jl,t}\|x_{l,t}-x_{j,t}\|+m\gamma (G_{\phi}+nG_f).
	\end{align}
	For the first term in \eqref{xsbus}, by the nonexpansiveness of the proximal operator,
	$$\|x_{l,t}-x_{j,t}\|\leq \|v_{l,t-1}-v_{j,t-1}\|.$$
	In addition, by \cite[Proposition 1]{distri-nedich}, {the bound of the distance between iterates $v_{j,t}$ and $\bar{v}_{t}$ satisfies}
	\begin{align}\label{vsubbarv}
		\|v_{j,t}-\bar{v}_{t}\|=&\Big\|\sum_{l=1}^m \big(\lambda_{jl,t} x_{l,t}^n-\frac{1}{m}x_{l,t}^n\big)\Big\|\notag\\
		\leq &\sum_{l=1}^m \big|\lambda_{jl,t}-\frac{1}{m}\big|\|x_{l,t}^n\|\notag\\
		\leq & \Gamma \Xi^t \sum_{l=1}^m\|x_{l,t}^n\|.
	\end{align}
	Then, by transformation,
	\begin{align}\label{vsubs}
		&\sum_{j=1}^m\sum_{l=1}^m \lambda_{jl,t}\|v_{l,t-1}-v_{j,t-1}\|\notag\\
		\leq & \sum_{j=1}^m\sum_{l=1}^m \lambda_{jl,t}(\|v_{l,t-1}-\bar{v}_{t-1}\|+\|v_{j,t-1}-\bar{v}_{t-1}\|)\notag\\
		\leq & 2m\Gamma \Xi^{t-1} \sum_{l=1}^m\|x_{l,t-1}^n\|.
	\end{align}
	Substituting \eqref{vsubs} to \eqref{xsbus},
	\begin{align*}
		&\sum_{j=1}^m \|x_{j,t+1}-x_{j,t}\|\notag\\
		\leq & 2m\Gamma \Xi^{t-1} \sum_{l=1}^m\|x_{l,t-1}^n\|+m\gamma (G_{\phi}+nG_{f}).
	\end{align*}
	\par (c) By the definition of $\bar{x}_t=\frac{1}{m}\sum_{p=1}^m x_{p,t}$, $\sum_{j=1}^m \|x_{j,t}-\bar{x}_t\|$ satisfies
	\begin{align*}
		&\sum_{j=1}^m\|x_{j,t}-\frac{1}{m}\sum_{p=1}^m x_{p,t} \|\\
		=&\sum_{j=1}^m\|\frac{1}{m}\sum_{p=1}^m(x_{j,t}-x_{p,t})\|\\
		\leq& \frac{1}{m}\sum_{j=1}^m\sum_{p=1}^m\|v_{j,t-1}-v_{p,t-1}\|\\
		\leq &\frac{1}{m}\sum_{j=1}^m\sum_{p=1}^m(\|v_{j,t-1}-\bar{v}_{t-1}\|+\|v_{p,t-1}-\bar{v}_{t-1}\|)\\
		\leq & 2 m \Gamma \Xi^{t-1} \sum_{l=1}^m\|x_{l,t-1}^n\|,
	\end{align*}
	where the first inequality follows from Proposition \ref{prox_proposition} (b) and the last inequality follows from \eqref{vsubbarv}.
\end{proof}

\par The following Lemma prove that the sequence $\sum_{j=1}^m \|x_{j,t}^n\|$ is upper bounded by a polynomial-geometric sequence, which is vital in the discussions of the summability of the error sequences.
\begin{lemma}\label{qbound_lemma}
	Under Assumptions \ref{f_assump} and \ref{net_assum}, there exist non-negative scalars $C_q=C_q(x_{1,1}^n,\cdots,x_{m,1}^n)$, $C_q^1=C_q^1(m,\Gamma,C_q,C_q^{2})$, $C_q^{2}=C_q^{2}(m,\gamma,G_g,G_h)$ such that for iteration $t\geq 1$,
	\begin{align*}
	\sum_{j=1}^m\|x_{j,t}^n\|\leq C_q+C_q^1 t+C_q^2 t^2.
	\end{align*}
\end{lemma}


\begin{proof}
	This proof is inspired by \cite[Lemma 1]{proximal-convex}. We proceed by induction on $t$. First, we show that the result holds for $t=1$ by choosing $C_q=\sum_{j=1}^m\|x_{j,1}^n\|$. It suffices to show that $\sum_{j=1}^m\|x_{j,1}^n\|$ is bounded given the initial points $x_{j,0}^n$.
	\par Indeed, by \eqref{xn},
	$$\sum_{j=1}^m\|x_{j,1}^n\|\leq \sum_{j=1}^m \|x_{j,0}^n\|+m\gamma (G_{\phi}+nG_f)<\infty.$$
	Therefore, $C_q=\sum_{j=1}^m\|x_{j,1}^n\|<\infty$ is a valid choice.
	\par Now suppose the result in Lemma \ref{qbound_lemma} holds for some positive integer $t\geq 1$. We show that it also holds for $t+1$. We transform \eqref{xn} in Lemma \ref{ite_seq} to 
	\begin{align}\label{scale_a}
	\sum_{j=1}^m\|x_{j,t+1}^n\|\!\leq\! \sum_{j=1}^m\! \|x_{j,t}^n\|\!+\!m\gamma (G_{\phi}\!+\! nG_f)\!+\!\sum_{j=1}^m \!\|x_{j,t}\!-\!x_{j,t-1}\|.
	\end{align}
	\par For the last term in \eqref{scale_a}, by transforming \eqref{ite_b} in Lemma \ref{ite_seq}, we obtain
	\begin{align}\label{nn1}
	\sum_{j=1}^m \!\|x_{j,t+1}\!-\!x_{j,t}\|\!
	\leq\!2m\Gamma\!\sum_{p=1}^{t-1}\!\Xi^{p}\!\sum_{j=1}^m\! \|x_{j,l}^n\|\!+\!(t\!-\!1) m \gamma (\!G_{\phi}\!+\! n G_f\!).
	\end{align}
	Then, substituting the induction hypothesis for $t$ into \eqref{nn1}, we have 
	\begin{align*}
		\sum_{j=1}^m \|x_{j,t+1}-x_{j,t}\|
		\leq& 2m\Gamma\sum_{p=1}^{t-1}\Xi^{p}(C_q+C_q^1 p+C_q^2 p^2)\\
		&+(t-1) m \gamma(G_{\phi}+n G_f).
	\end{align*}
   By Lemma \ref{poly_lemma} and \eqref{kn_poly}, there exist constants $S_{0}^{\Xi}, S_{1}^{\Xi}, S_2^{\Xi}$ such that 
	$$\sum_{p=1}^{\infty} \Xi^p (C_q+C_q^1 p+C_q^2 p^2)\leq C_q S_0^{\Xi}+C_q^1S_1^{\Xi}+C_q^2 S_2^{\Xi}.$$
 By induction hypothesis and \eqref{scale_a},
	\begin{align}\label{xtplus}
		\sum_{j=1}^m \|x_{j,t+1}^n\|\leq& C_q+C_q^1 t+C_q^2 t^2+t m \gamma (G_{\phi}+nG_f)\notag\\
		&+2m\Gamma (C_q S_0^{\Xi}+C_q^1S_1^{\Xi}+C_q^2 S_2^{\Xi}).
	\end{align}
 Take the coefficients $C_q$, $C_q^1$ and $C_q^2$ as
	\begin{align*}
		C_q&=\sum_{j=1}^m \|x_{j,1}^n\|,\\
		C_q^1&=\frac{2m\Gamma C_q S_0^{\Xi}+(2m\Gamma S_2^{\Xi}-1)C_q^2}{2m\Gamma S_1^{\Xi}-1},\\
		C_q^2&=\frac{\gamma m}{2}(G_{\phi}+nG_f).
	\end{align*}
Then, comparing coefficients, we see that the right-hand side of \eqref{xtplus} has an upper bound $C_q+C_q^1 (t+1)+C_q^2(t+1)^2$, which implies that the induction hypothesis holds for $t+1$.
\end{proof}
\textbf{Proof of Proposition \ref{e_sum_pro}:}
\begin{proof}
	\par (a) By \eqref{ite_c} in Lemma \ref{ite_seq},
	\begin{align}\label{xsubbar}
		\|x_{j,t}-\bar{x}_{t} \|
		\leq  2  \Gamma \Xi^{t-1} \sum_{l=1}^m\|x_{l,t-1}^n\|.
	\end{align} 
	Then, by Lemma \ref{qbound_lemma}, \eqref{norm_e} and \eqref{xsubbar}, there exist $C_q$, $C_q^1$, $C_q^2>0$ such that 
		\begin{align}\label{bet}
			{\|e_t^i\|}
			&\leq \underbrace{2 L \Gamma \Xi^{t-1} (C_q\!+\!C_q^1 (t-1)\!+\!C_q^2 (t-1)^2)}_{b_{e,t}}+\gamma C_0,
		\end{align}
	where $C_0=2 L n G_f$ and $b_{e,t}$ is a polynomial-geometric sequence.
	\par (b) It follows from (\ref{norm_vare}), \eqref{vsubbarv} and Lemma \ref{qbound_lemma} that 
	\begin{align}\label{bvapt}
		&\varepsilon_{t+1}\leq  2G_{\phi}\Gamma \Xi^t \sum_{l=1}^m \|x_{l,t}^n\|+\frac{1}{2\gamma}(\Gamma \Xi^t \sum_{l=1}^m \|x_{l,t}^n\| )^2\notag\\
		& \leq\! \underbrace{\!2G_{\phi}\Gamma \Xi^t (C_q\!+\!C_q^1 t\!+\!C_q^2 t^2)\!+\!\frac{1}{2\gamma}\!\big[\Gamma \Xi^t (C_q\!+\!C_q^1 t\!+\!C_q^2 t^2) \big]^2\!}_{b_{\varepsilon,t}}.
	\end{align}
	Using the fact that $\sqrt{a+b}\leq \sqrt{a}+\sqrt{b}$ for all nonnegative real numbers $a,b$, 
	\begin{align}\label{sqrtbvapt}
		\sqrt{\varepsilon_{t+1}}\!\leq\! & \underbrace{\!\sqrt{2G_{\phi} \Gamma} \sqrt{\Xi^t}\big(\sqrt{C_q}\!+\!\sqrt{C_q^1t}\!+\!\sqrt{C_q^2}t\big)\!+\!\frac{1}{\sqrt{2\gamma}} \Gamma \Xi^t C(t)\!}_{b_{\sqrt{\varepsilon},t}},
	\end{align} 
	where $C(t)=(C_q+C_q^1 t+C_q^2 t^2)$.
\end{proof}
\subsection{Proof of Lemma \ref{in_V_lem}}\label{Vt_proof}
\begin{proof}
	By Lemma \ref{pk_lemma} and \eqref{xave_up}, we have
	\begin{align}\label{xd}
	\bar{x}_t-\gamma \bar{g}_t -\bar{x}_{t+1}-\gamma \underbrace{\sum_{i=0}^{n-1} e_{t}^i}_{e_{t+1}}-p_{t+1}= \gamma \bar{d}_{t+1},
	\end{align}
	where $\|p_{t+1}\|\leq \sqrt{2 \gamma \varepsilon_{t+1}}$, $\bar{d}_{t+1}\in \partial_{\varepsilon_{t+1}} \phi (\bar{x}_{t+1})$, $\bar{g}_t\triangleq \sum_{i=0}^{n-1} \nabla \mathbf{f}_{\pi_t^i}(\bar{x}_t^i)$ and $\mathbf{f}_{\pi_t^i}(x)=\frac{1}{m}\sum_{j=1}^m  f_{j,\pi_t^i}(x)$.
	\par In addition, by \eqref{barx_up}, $\bar{x}_t^k=\bar{x}_t-\gamma \sum_{i=0}^{k-1} \nabla \mathbf{f}_{\pi_t^i}(\bar{x}_t^i)-\gamma \sum_{i=0}^{k-1} e_t^i$. Then, combining \eqref{xd} with \eqref{barx_up},
	\begin{align*}
	&\bar{x}_t^k-\bar{x}_{t+1}\\
	=&\bar{x}_t-\gamma \sum_{i=0}^{k-1} \nabla \mathbf{f}_{\pi_t^i}(\bar{x}_t^i)-\gamma \sum_{i=0}^{k-1} e_t^i\\ &-\big(\bar{x}_t-\gamma \sum_{i=0}^{n-1} \nabla \mathbf{f}_{\pi_t^i}(\bar{x}_t^i)-\gamma \sum_{i=0}^{n-1} e_{t}^i-p_{t+1}-\gamma  \bar{d}_{t+1}\big)\\
	=&\gamma \sum_{i=k}^{n-1}  \nabla \mathbf{f}_{\pi_t^i}(\bar{x}_t^i) +\gamma \underbrace{\sum_{i=k}^{n-1} e_t^i}_{\tilde{e}_{k,t+1}}+p_{t+1}+\gamma  \bar{d}_{t+1}.
	\end{align*}
	Taking the norm of $\bar{x}_t^k-\bar{x}_{t+1}$, we have
	\begin{align}\label{barxnorm}
	&\|\bar{x}_t^k -\bar{x}_{t+1}\|^2\notag\\
	=&\|\gamma \sum_{i=k}^{n-1}  \nabla \mathbf{f}_{\pi_t^i}(\bar{x}_t^i) +\gamma \tilde{e}_{k,t+1}+p_{t+1}+\gamma\bar{d}_{t+1}\|^2\notag\\
	\leq & 6\gamma^2 \|\!\sum_{i=k}^{n-1} \! \nabla \mathbf{f}_{\pi_t^i}(\bar{x}_t^i)\|^2\!+\!6\gamma^2\|\tilde{e}_{k,t+1}\|^2\!\notag\\
	&+\!6\gamma^2\|\bar{d}_{t+1}\|^2\!+\!2\|p_{t+1}\|^2\notag\\
	\leq & 6\gamma^2\! \Big(\!2\big\|\sum_{i=k}^{n-1} \! (\nabla \mathbf{f}_{\pi_t^i}\!(\bar{x}_t^i)\!-\!\nabla \mathbf{f}_{\pi_t^i}\!(x_*))\big\|^2\!+\!2\big\|\sum_{i=k}^{n-1}\!\nabla \mathbf{f}_{\pi_t^i}\!(x_*) \big\|^2\!\Big)\notag\\
	&+6\gamma^2 \|\tilde{e}_{k,t+1}\|^2+6\gamma^2 \|\bar{d}_{t+1}\|^2+2\|p_{t+1}\|^2\notag\\
	\overset{\eqref{fnablf},\eqref{pt_range}}{\leq} & 6\gamma^2\! \big(4Ln \!\sum_{i=k}^{n-1} D_{\mathbf{f}_{\pi_t^i}}(x_*,\bar{x}_t^i)+2\|\sum_{i=k}^{n-1}\nabla \mathbf{f}_{\pi_t^i}(x_*) \|^2\big)\notag\\
	&+6\gamma^2 \|\tilde{e}_{k,t+1}\|^2+6\gamma^2 \|\bar{d}_{t+1}\|^2+4\gamma \varepsilon_{t+1}\notag \\
	\overset{\eqref{bet},\eqref{bvapt}}{\leq} & 6\gamma^2 \big(4Ln \sum_{i=0}^{n-1} D_{\mathbf{f}_{\pi_t^i}}(x_*,\bar{x}_t^i)\!+\!2\|\sum_{i=k}^{n-1}\nabla \mathbf{f}_{\pi_t^i}(x_*) \|^2\big)\notag\\
	&+12n^2\gamma^2  b_{e,t}^2 + 12 n^2 \gamma^4  C_0^2+6\gamma^2 \|\bar{d}_{t+1}\|^2+4\gamma b_{\varepsilon,t}.
	\end{align}
	{Next, consider the term $\|\sum_{i=k}^{n-1}\nabla \mathbf{f}_{\pi_t^i}(x_*) \|^2$ in \eqref{barxnorm}. By Lemma \ref{xavelem} with $\bar{X}_{\pi}=\frac{1}{n-k} \sum_{i=k}^{n-1} \nabla \mathbf{f}_{\pi_t^i}(x_*)$ and $\bar{X}=\frac{1}{n}\nabla f(x_*)$, 
		\begin{align*}
		&\mathbb{E} \Big[\| \sum_{i=k}^{n-1}\nabla \mathbf{f}_{\pi_t^i}(x_*) \|^2\Big]\\
		=&(n-k)^2 \mathbb{E}\Big[ \|\frac{1}{n-k} \sum_{i=k}^{n-1} \nabla \mathbf{f}_{\pi_t^i}(x_*)\|^2\Big]\\
		=& (n-k)^2 \mathbb{E}\big[\|\bar{X}_{\pi}\|^2\big]\\
		= &(n-k)^2 \big( \|\bar{X}\|^2+\mathbb{E}[\|\bar{X}_{\pi}-\bar{X}\|^2]\big)\\
		=& (n-k)^2 \|\frac{1}{n}\nabla f(x_*) \|^2+(n-k) \frac{k}{(n-1)} \sigma_*^2,
		\end{align*}
		where $\sigma_*^2\triangleq \frac{1}{n}\sum_{i=1}^{n} \|\nabla \mathbf{f}_{\pi_t^i}(x_*)-\frac{1}{n}\nabla f(x_*)\|^2$ is the population variance at the optimum $x_*$. } 
	Summing this over $k$ from $0$ to $n-1$,
	\begin{align}\label{sumnablaf}
	&\sum_{k=0}^{n-1} \mathbb{E} \Big[\| \sum_{i=k}^{n-1}\nabla \mathbf{f}_{\pi_t^i}(x_*) \|^2\Big]\notag\\
	\leq& \sum_{k=0}^{n-1} (n-k)^2 \|\frac{1}{n}\nabla f(x_*) \|^2+\sum_{k=0}^{n-1} \frac{k(n-k)}{n-1}\sigma_*^2\notag\\
	=&\frac{n(n+1)(2n+1)}{6n^2} \|\nabla f(x_*) \|^2+ \frac{n(n+1)}{6}\sigma_*^2\notag\\
	\leq & \frac{3n}{4} G_f^2 +\frac{n^2\sigma_*^2}{4}, 
	\end{align}
	where the inequality holds due to $n\geq 2$, $n+1<2n$ and Assumption \ref{f_assump} (d). 
	\par Thus, summing \eqref{barxnorm} over $k$ gives
	\begin{align*}
	&\sum_{k=0}^{n-1}\mathbb{E}[\|\bar{x}_t^k-\bar{x}_{t+1}\|^2]\\
	\leq &24\gamma^2Ln^2\sum_{i=0}^{n-1} \mathbb{E}[D_{\mathbf{f}_{\pi_t^i}}(x_*,\bar{x}_t^i)]+9\gamma^2 n G_f^2+3\gamma^2 n^2\sigma_*^2\\
	&+6\gamma^2 n G_{\phi}^2+12n^3\gamma^2  b_{e,t}^2 + 12 n^3 \gamma^4  C_0^2+4n\gamma b_{\varepsilon,t},
	\end{align*}
	where we use \eqref{sumnablaf} and Assumption \ref{f_assump} (c).
\end{proof}
\bibliographystyle{ieeetran}
\bibliography{refer}

\end{document}